\documentclass[aap,authoryear]{imsart}

\RequirePackage[OT1]{fontenc}
\RequirePackage{amsthm,amsmath}
\RequirePackage[numbers]{natbib}
\usepackage{algorithm}
\newcommand\asp{\hspace{4mm}}

\newcommand{\sE}{{\mathcal{E}}}

\usepackage{graphicx}
\usepackage{mathrsfs}
\usepackage{amsfonts}
\usepackage{threeparttable,booktabs}
\usepackage{enumerate}
\usepackage{tikz}
\usepackage{listings}
\usepackage{multirow}
\usepackage{amsfonts}
\usepackage{amssymb}
\usepackage{bbm}
\usepackage{arydshln}

\allowdisplaybreaks

\makeatletter
\def\thm@space@setup{%
  \thm@preskip=0.45cm
  \thm@postskip=\thm@preskip 
}
\makeatother

\startlocaldefs
\numberwithin{equation}{section}
\theoremstyle{plain}
\newtheorem{thm}{Theorem}[section]
\newtheorem{definition}[thm]{Definition}

\newtheorem{corollary}[thm]{Corollary}
\newtheorem{example}[thm]{Example}

\newtheorem{lemma}[thm]{Lemma}
\theoremstyle{remark}

\endlocaldefs

\def\card{\mathrm{card}}
\def\<{{\langle}}
\def\>{{\rangle}}

\usepackage{enumitem}
\setlist{  
	listparindent=\parindent,
	parsep=0pt,
}

\usepackage{smartdiagram}

\usepackage{subfig}
\usepackage{wrapfig}
\usepackage{color}

\begin{document}

\begin{frontmatter}
\title{Convergence of Dirichlet Forms for MCMC Optimal Scaling with Dependent Target Distributions on Large Graphs}
\runtitle{Convergence of Dirichlet Forms for MCMC Optimal Scaling}

\begin{aug}

%

\author{\fnms{Ning} \snm{Ning}\thanksref{e1}\ead[label=e1,mark]{patning@tamu.edu}}

\address{Department of Statistics,
Texas A\&M University, College Station, Texas. \\
\printead{e1}}

\runauthor{Ning Ning}


\end{aug}
\begin{abstract}
Markov chain Monte Carlo (MCMC) algorithms have played a significant role in statistics, physics, machine learning and others, and they are the only known general and efficient approach for some high-dimensional problems. The random walk Metropolis (RWM) algorithm as the most classical MCMC algorithm, has had a great influence on the development and practice of science and engineering. The behavior of the RWM algorithm in high-dimensional problems is typically investigated through a weak convergence result of diffusion processes. In this paper, we utilize the Mosco convergence of Dirichlet forms in analyzing the RWM algorithm on large graphs, whose target distribution is the Gibbs measure that includes any probability measure satisfying a Markov property. The abstract and powerful theory of Dirichlet forms allows us to work directly and naturally on the infinite-dimensional space, and our notion of Mosco convergence allows Dirichlet forms associated with the RWM  chains to lie on changing Hilbert spaces. Through the optimal scaling problem, we demonstrate the impressive strengths of the Dirichlet form approach over the standard diffusion approach.
\end{abstract}


%

\end{frontmatter}

\tableofcontents

\section{Introduction}
\label{sec:Introduction}

We firstly give the background and motivation in Section \ref{sec:Background}, and then summarize our contributions in Section \ref{sec:Our_contributions}, followed by the organization of the paper in Section \ref{sec:Organization}.

\subsection{Background and motivation}
\label{sec:Background}

Markov chain Monte Carlo (MCMC) algorithms are generally used for sampling from multi-dimensional distributions, especially when the number of dimensions is high.
They have played
a significant role in statistics, physics, machine learning and others, and are the only known general and efficient approach for some high-dimensional problems \citep{andrieu2003introduction}. In practice, exact inference is frequently intractable, so approximate inference is often performed using MCMC.
The random walk Metropolis (RWM) algorithm, as one of the
simplest and most popular MCMC algorithms, works by generating a sequence of sample values from a probability distribution from which direct sampling is difficult, such that the distribution of values more closely approximates the desired distribution as more sample values are produced. Specifically, these sample values are produced iteratively, in such a way that the algorithm picks a candidate for the next sample value based on the current sample value at each iteration. With some probability, the candidate is either accepted to be the value in the next iteration, or rejected and then the current value is reused in the next iteration. It works provided that we know a function proportional to the desired density and the values of that function can be calculated, which in turn makes the MCMC algorithm particularly useful since calculating the necessary normalization factor is often extremely difficult in practice (see, e.g. \cite{medina2016stability, atchade2007geometric}).

Optimal scaling analysis for MCMC algorithms is one of the most successful and practically useful way of performing asymptotic analysis in high-dimensions \citep{gelman1997weak,roberts1998optimal,roberts2001optimal}, considering the case that the dimension of the target distribution ($n$) goes to infinity. It investigates $\sigma_n$ which is the variance of the random walk step, aiming to answer two questions: How should $\sigma_n$ scale as a function of $n$, so as to optimize the speed of convergence of the algorithm in
the limit? Is it possible to characterize the optimality of $\sigma_n$ in a way that can be practically utilized? 
\cite{gelman1997weak} showed that when the proposal variance is appropriately scaled according
to $n$, the sequence of stochastic processes formed by the first component
of each Markov chain converges to the appropriate limiting Langevin
diffusion process as $n$ goes to infinity.
This limiting diffusion approximation admits a straightforward efficiency
maximization problem, leading to an asymptotically optimal acceptance rate $0.234$.

As the first optimal scaling paper of continuous state spaces, \cite{gelman1997weak} considered the case that the target distribution $(\pi_n)$ is in an independent and identically distributed (i.i.d.) product form, i.e., $\pi_n(x^n)=\prod_{i=1}^n f(x_i)$. \cite{gelman1997weak} assumed that $f'/f$ is Lipschitz continuous, the existence of finite fourth moment of $f''/f$, and the existence of finite eighth moment of $f'/f$. 
	Later, \cite{durmus2017optimal} extended the optimal scaling analysis to include densities that are differentiable in the $L^p$ mean but may exhibit irregularities at certain points, improving the above eighth moment constraint to the sixth moment, along with other enhancements.
 \cite{zanella2017dirichlet} introduced Mosco convergence \citep{mosco1992composite} of Dirichlet forms to the optimal scaling problem for the RWM algorithm. The Dirichlet form approach allows one to replace the restrictive conditions imposed in \cite{gelman1997weak} with substantially weaker ones. Specifically,  with product-formed target distributions \cite{zanella2017dirichlet} merely assumed that $f'/f$ satisfies a combined growth/local H{\"o}lder condition and the existence of finite second moment of $f'/f$.

 \cite{zanella2017dirichlet} embedded the finite $n$-dimensional Markov chain into the infinite-dimensional space $\mathbb{R}^{\infty}$, updating the first $n$ components while drawing the remaining components independently from the target distribution and then holding them fixed. 
  They proved Mosco convergence of Dirichlet forms associated with the resulting infinite-dimensional Markov chains to the limiting Dirichlet form, as $n$ goes to infinity. 
  The following concerns arose.  Firstly, with i.i.d. product-formed target density, extending the finite-dimensional chain to an infinite dimension is not necessary. 
 In fact, it can be tackled by sampling a single one-dimensional target thanks to the product structure (page $6095$, \cite{yang2020optimal}). Secondly, probability distributions may not be finite in infinite-dimension, such as the Gibbs measure (see detailed explanations in Section \ref{sec:Gibbs_measures}). Thirdly, the i.i.d. form is too limited (page $6095$, \cite{yang2020optimal}).
After \cite{gelman1997weak} optimal scaling analysis focuses on non-(i.i.d.) product-formed target density (see, e.g. \cite{Breyer2000From, Neal2006Optimal,bedard2007weak, bedard2008optimal, beskos2009optimal, mattingly2012diffusion, Hairer2014Spectral, yang2020optimal}). Hence, whether Mosco convergence of Dirichlet forms is useful in modern optimal scaling analysis is an open question.  

The aim of this paper is to address this question. To illustrate the strength of Dirichlet forms and demonstrate broad applicability, it is essential to identify a seminal work in the existing literature that focuses on dependent distributions. One such work is \cite{Breyer2000From}, which is notable for being the first to consider a general distribution on a graph structure, specifically the Gibbs measure on $\mathbb{Z}^d$. According to the Hammersley–Clifford theorem, any probability measure that satisfies a Markov property can be represented as a Gibbs measure with an appropriate choice of energy function. Consequently, the framework introduced in \cite{Breyer2000From} is highly versatile and could be applied to various problems, including Hopfield networks, Markov networks, Markov logic networks, and boundedly rational potential games in game theory and economics, among others. Subsequently, \cite{yang2020optimal} explored dependent models on graphs; however, their mathematical formulations lack the generality of the Gibbs measure. Moreover, the regularity conditions imposed by \cite{yang2020optimal} are comparable to those established by \cite{Breyer2000From}.

\subsection{Summary of our contributions}
\label{sec:Our_contributions}

In this paper, we present proof strategies that leverage Mosco convergence of Dirichlet forms on evolving Hilbert spaces for optimal scaling analysis, involving Gibbs measures on large graphs. The specific contributions of the paper can be categorized into four main areas:

\begin{figure*}[t!]
	\centering
	{\includegraphics[width = 1.7in]{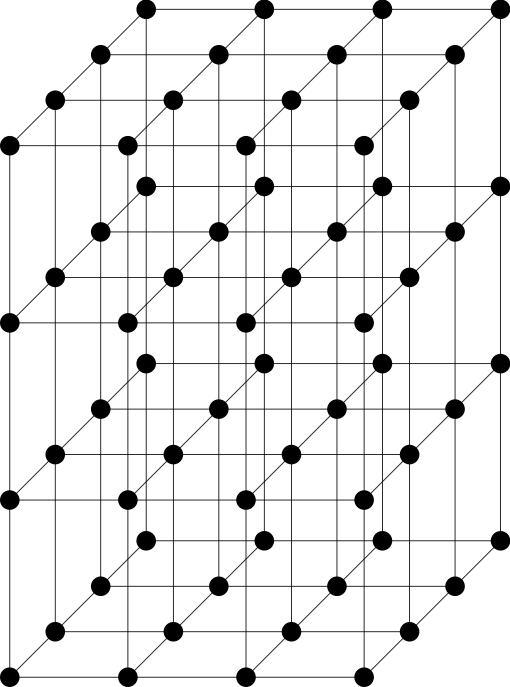}}\hfil
	{\includegraphics[width =2.2in]{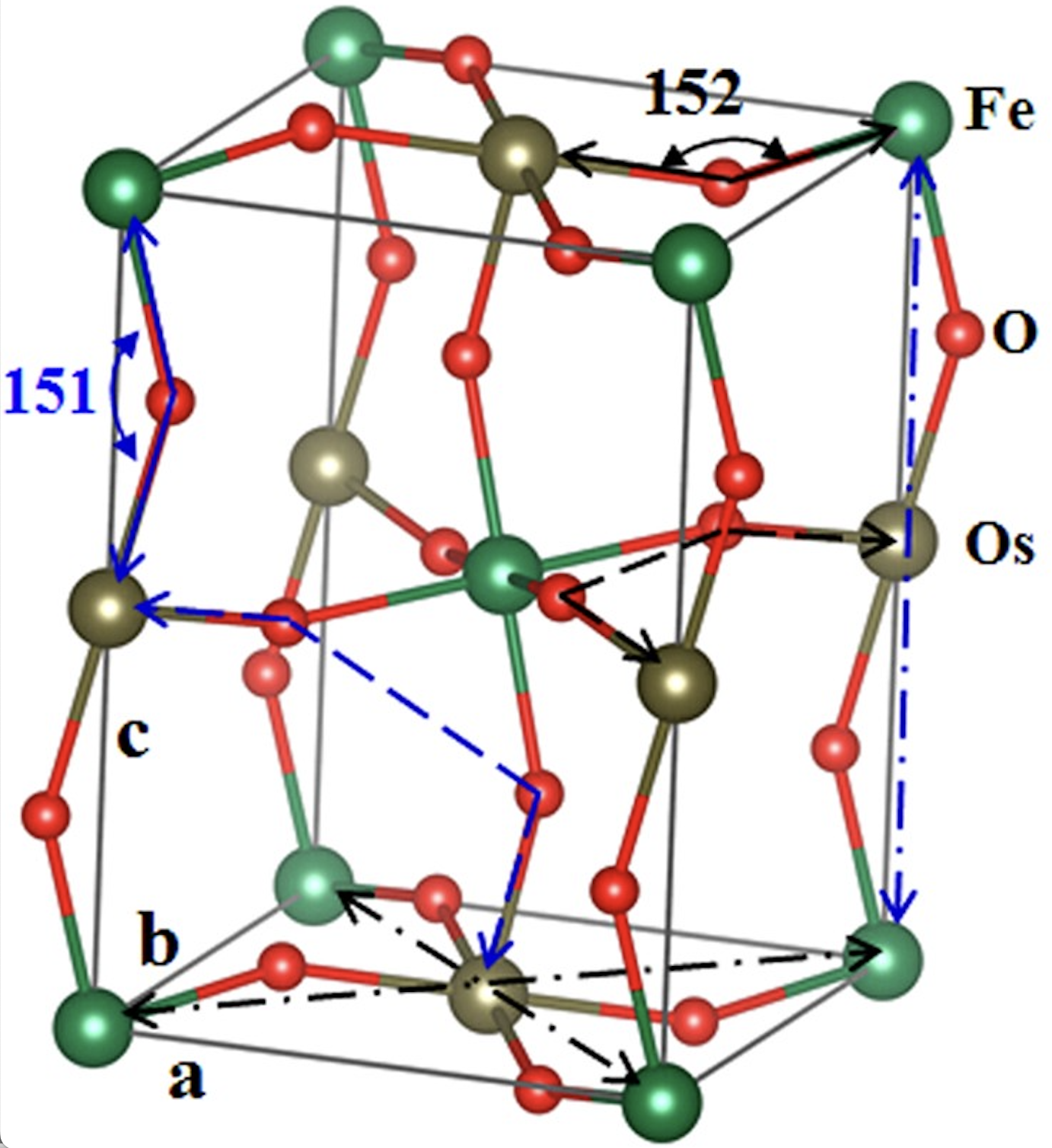}}
	\caption{Illustration of generalized additive property. The left figure depicts the typical lattice structure, while the right figure from \cite{hou2015lattice} showcases a distorted lattice structure in a real scientific application.}
	\label{fig:additive_illustration}
\end{figure*}
\begin{enumerate}
		\setlength\itemsep{0.3em}
\item \textbf{Dependent distribution on a large graph.} 
To ensure fair comparisons, we utilize the exact same model setup as \cite{Breyer2000From}, with the only difference being a more general graph structure. We study the Gibbs measure on a large graph, where each vertex corresponds to a real-valued random variable. The graph may be finite or infinite. It's worth noting that  for finite systems, Gibbs measure, Gibbs distribution, or Boltzmann distribution are used interchangeably, whereas one can only use the terminology Gibbs measure for infinite systems. 
Our graph structure is endowed with a generalized additive property, as depicted in Figure \ref{fig:additive_illustration}. By defining additivity in terms of the graph distance, we introduce flexibility, as demonstrated in the right figure of Figure \ref{fig:additive_illustration}, which illustrates the lattice-distortion used in \cite{hou2015lattice} on induced magnetic transition from low-temperature antiferromagnetism to high-temperature ferrimagnetism.

\item \textbf{Definitions on converging Hilbert spaces.} In contrast to the Mosco convergence of Dirichlet forms on the same space introduced in \cite{zanella2017dirichlet}, we introduce the  Mosco convergence of Dirichlet forms on changing Hilbert spaces. 
 Specifically, we work on a sequence of Hilbert spaces $\{\mathcal{H}_n\}_{n\in \mathbb{N}}$ with  $\mathcal{H}_n= L^2(\mathbb{R}^{V_n};\pi_n)$ converging to another Hilbert space $\mathcal{H}= L^2(\mathbb{R}^V;\pi)$ in the sense of Definition \ref{def:Hiblertspaces_cvg}, where $\{V_n\}_{n\in \mathbb{N}}$ is an exhausting sequence of subsets of the vertex set $V$. We identify the Dirichlet form $\mathcal{E}^n$ on $\mathcal{H}_n$ that is associated with the $n$-dimensional Markov chain and the limiting Dirichlet form $\mathcal{E}$ on $\mathcal{H}$. We use the definition of Mosco convergence on changing Hilbert spaces (Definition $2.11$, \cite{kuwae2003convergence}), which is an improved version of the original definition of Mosco convergence on a fixed Hilbert space (Definition $2.1.1$, \cite{mosco1992composite}).

\item \textbf{Technical contributions.} 
Our proof strategy differs significantly from that of \cite{zanella2017dirichlet}, primarily due to differences in definitions used and the distributions focused on. With the i.i.d. product-formed target distribution, their proof started with a logarithmic transformation that generated a desired summation form, and ended using the inequality that $|(1\wedge ab)-(1\wedge b)|\leq |1-a|$ for any $a,b>0$ to separate the components in their target distribution. However, that  approach is not applicable to non-product-formed target distributions. Consequently, we must devise strategies capable of addressing changing Hilbert spaces and general distributions simultaneously. Given the highly technical nature of the proofs, we outline the encountered challenges and provide proof sketches in Sections \ref{sec:Proofsa_cvg}, \ref{sec:ProofsM1_sketch}, and \ref{sec:ProofsM2_sketch}, before presenting the formal proofs of Theorem \ref{thm:a_convergence} and Conditions (M$1$) and (M$2$), respectively. The fulfillment of these two conditions in Definition \ref{def:Mosco_converges} achieves the Mosco convergence of $\mathcal{E}_n$ on $\mathcal{H}_n$ to $\mathcal{E}$ on $\mathcal{H}$.

\item \textbf{Strength of the Dirichlet form approach.}  There are two major advantages to employing Dirichlet form approaches. First, it usually requires only modest regularity assumptions. 
In the seminal work of \cite{Breyer2000From}, Hypotheses (H$1$)-(H$6$) were assumed in Theorem \ref{thm:classical_results}. We adopt their Hypothesis (H$2$) alongside our Assumptions (A$1$) and (A$2$). Specifically, Assumption (A$1$) broadens the scope of models considered compared to Hypothesis (H$1$) by relaxing assumptions, removing the requirement for the target density function to possess continuous third-order derivatives. Assumption (A$2$) eliminates the need for bounded third-order derivatives as stipulated in Hypothesis (H$3$).  We remove Hypotheses (H$4$), 
(H$5$), and (H$6$). See Examples \ref{example:no_3rd}, \ref{example3}  and \ref{example4} for instances that do not satisfy Hypotheses (H$3$), (H$4$) and (H$5$), respectively. 
No example satisfying Hypothesis (H$6$) was provided in \cite{Breyer2000From}, which is difficult or even impossible to verify.

Second, the abstract and powerful theory of Dirichlet forms, particularly the concept of Mosco convergence, enables direct work on infinite-dimensional spaces. With the diffusion approach, \cite{Breyer2000From} could only establish the existence of the limiting diffusion on the state space $L^2(\mathbb{R}^{\mathbb{Z}^d}; \pi)$, which should be $\mathbb{R}^{\mathbb{Z}^d}$.  In fact, the infinite-dimensional space $\mathbb{R}^{\mathbb{Z}^d}$ is not a Banach space because its topology cannot be derived from any norm. They have to reduce the state space to a Hilbert space. Such complications are naturally circumvented with the Dirichlet form approach.
\end{enumerate}

\subsection{Organization of the paper}
\label{sec:Organization}
The remaining sections of the paper unfold as follows. In Section \ref{sec:Optimalscaling}, we delve into the Gibbs measure and the RWM sampler, while also summarizing the existing optimal scaling results.
Our main results are outlined in Section \ref{sec:mainresults}, encompassing the necessary definitions tailored to changing Hilbert spaces and Mosco convergence of Dirichlet forms.
All technical proofs are deferred to Section \ref{sec:Technical_proofs}, accompanied by comments on the encountered challenges and proof outlines.
Section \ref{sec:Discussion_and_conclusion} offers a discussion and conclusion for this paper. In Appendix \ref{sec:Numerical}, we present a numerical demonstration of how to apply the optimal scaling results obtained in this paper to a practical example that does not satisfy the assumptions enforced in \cite{Breyer2000From}.

\section{Optimal scaling for the RWM algorithm}
\label{sec:Optimalscaling}
In this section, we first describe the Gibbs measure
in Section \ref{sec:Gibbs_measures} and the RWM sampler to learn the Gibbs measure
in Section \ref{sec:MHRW}, summarize the
existing weak convergence results using the diffusion approach in \cite{Breyer2000From}
in Section \ref{sec:existing_results}, and at last illustrate the resulting optimal scaling for the RWM algorithm 
in Section \ref{sec:Optimal_scaling}.

\subsection{Gibbs measures}
\label{sec:Gibbs_measures}
The Gibbs measure is a generalization of the canonical ensemble to infinite systems, where a canonical ensemble is the statistical ensemble that represents the possible states of a mechanical system. The statistical ensemble gives the probability of the system $X$ being state $x$ as
\begin{equation}
	\label{eqn:exponential_family}
\pi(dx)=\frac{1}{Z_P}\exp(-H(x))\mu(dx).
\end{equation}
Here, $H(x)$ is the function from the space of states to the real numbers. It is called the Hamiltonian or the energy function in statistical physics, representing the energy of the configuration $x$. The normalizing constant $Z_P$ is called the partition function and $\mu$ is a reference measure on the infinite system. 
We consider the Gibbs random field on a large graph $G=(V,E)$ where $V$ is a countable set
of vertices and $E$ is a set of edges. The countable set $V$ is considered to be located on the $d$-dimensional Euclidean space $\mathbb{R}^d$.
Here, $G$ can be a finite graph such that the cardinality of the vertex set $|V|<\infty$ or an infinite graph such that $|V|=\infty$. In this paper, we focus on the challenging infinite case which has the $d$-dimensional integer lattice $\mathbb{Z}^d$ as a special case. Each vertex $k\in V$ is given by a real-valued random variable $X_k$ whose realization is $x_k$. 

Although the Gibbs measure was proposed to provide a framework to directly study infinite systems, the energy function may not be finite in infinite systems. Hence, the Gibbs measure is defined in terms of the conditional probabilities that it induces on each finite subsystem. We consider $\{V_n\}_{n\in \mathbb{N}}$ as an exhausting sequence of subsets such that $V_n\rightarrow  V$ as $n\rightarrow \infty$.  We denote $x_{V_n}=\{x_k\}_{k\in V_n}$ and call it a configuration on $V_n$. Similarly, we let $x_{V_n^c}=\{x_k\}_{k\in V\backslash V_n}$. Then for a configuration $x\in \mathbb{R}^{V}$ on the infinite system, we can write $x=(x_{V_n^c},x_{V_n})$. Correspondingly, we define $\sigma$-algebras $\mathcal{F}_{V_n}=\sigma(x_{V_n})$ and $\mathcal{F}_{V_n^c}=\sigma(x_{V_n}^c)$. Instead of analyzing the possibly infinite Hamiltonian $H(x)$ on the infinite system, the Gibbs measure focuses on the Hamiltonian $H_{W}(z_{W^c},x_{W})$ restricted on any finite subsystem $W\subset  V$, conditional on some fixed boundary configuration $z_{W^c}$. In this paper, we consider the reference measure as the Lebesgue measure and set $dx_{W}=\prod_{k\in W}dx_k$.
Define
\begin{equation}
\label{eqn:big_pi_0}
\pi_{W}(z,dx)=\frac{1}{Z_{W}}\exp(-H_{W}(z_{W^c},x_W))dx_{W}.\end{equation}
Then the Gibbs measure is defined as the family
\begin{equation}
\label{eqn:big_pi}
\Pi=\Big\{\pi_{W}(z,dx):  W\subset  V\text{ being finite and } z\in \mathbb{R}^V\Big\}
\end{equation}
satisfying the following consistency condition:
$$\int\pi_{W}(z,dy)\pi_{U}(y,dx)=\pi_{W}(z,dx), \qquad U\subset W\subset  V.$$

The Gibbs measure is very broad. The Hammersley–Clifford theorem implies that any probability measure that satisfies a Markov property is a Gibbs measure for an appropriate choice of Hamiltonian. Hence, proper assumptions are usually imposed. We follow \cite{Breyer2000From} in adopting the following two restrictions, both of which are often satisfied in the statistical analysis of certain spatial models. First, the Gibbs measure is assumed to be transition invariant, i.e., $\pi\circ\oplus_l=\pi$ for all $l\in  V$ where $\oplus_l$ denotes the shift transformation $\oplus_l x_j=x_{j+l}$ for $j\in  V$ and $x\in \mathbb{R}^V$. Second, we consider the Gibbs measure that has finite-range interactions. 
That is, for any finite set $W\subset  V$, the Hamiltonian $H_{W}$ with boundary condition $z_{W^c}\in \mathbb{R}^V$ is the function
\begin{equation}
H_{W}(z_{W^c},x_W) =-\sum_{k\in W} h_k(z_{W^c},x_W),\qquad x\in \mathbb{R}^V,
\end{equation}
where $h_k$ is a function of a finite number of components. 
Then by equation \eqref{eqn:big_pi_0}, the probability measure $\pi_W(z, dx)$ is given by
$$\pi_{W}(z,dx)=\frac{1}{Z_{W}}\exp\left(\sum_{k\in W} h_k(z_{W^c},x_W)\right)dx_{W}.$$
Throughout the paper, the dependence on the boundary condition shall be ignored for simplicity when no confusion is raised. 

\subsection{RWM sampler}
\label{sec:MHRW}

Most probability distributions $\pi_n (=\pi_{V_n})$ on $\mathbb{R}^{V_n}$ can be approximated by the discrete-time Markov chain $X^n$ generated by the RWM algorithm. The Markov chain has $\pi_n$ as its equilibrium distribution. The states of the chain after deleting the initial ``burn-in" steps are then used as samples from $\pi_n$.
That is, if the initial state $X^n(0)$ is not distributed according to $\pi_n$, mixing time analysis is involved which analyzes the limiting behavior of the chain as $t\rightarrow \infty$. However, the optimal scaling problem is in another perspective that focuses on the asymptotic analysis as $n\rightarrow \infty$. Hence, in optimal scaling literature, initial values are mainly assumed to be distributed as $\pi_n$, which will be the case in this paper. Given that $X^n$ is time-homogeneous, we only need to describe its first step while the remaining steps follow immediately from its Markov property.  

Generate the $n$-dimensional initial value $x^n(0)$ according to $\pi_n$. Let $U$ be a uniform random variable on $[0,1]$. Suppose that $R=(R_i)_{i=1}^{\infty}$ is a sequence of i.i.d. random variables having a symmetric probability density function (PDF) $\varphi(\cdot)$ and unit variance. 
We note that the PDF of the normal distribution is a popular choice of $\varphi(\cdot)$, but we do not restrict it here.
Denote $R^n=(R_1,\ldots,R_n)$. We let $X^n(0)$, $R^n$ and $U$ be independent of each other. Denote the density of $\pi_n$ (also known as the Radon–Nikodym derivative) as 
$$\psi_n(x)=\pi_n(dx)/dx_{V_n}.$$
The first step $X^n(1)$ is generated as follows
\begin{equation}
	\label{eqn:X1X0}
	X^n(1)=X^n(0)+A_n\sigma_n R^n,
\end{equation}
where $A_n=1$ if $U<a_n(X^n(0), X^n(0)+\sigma_nR^n)$ and  $A_n=0$ otherwise, with
\begin{equation}
\label{eqn:def_acceptance_probability}
a_n\big(X^n(0), X^n(0)+\sigma_nR^n\big) :=1\wedge\dfrac{\psi_n\big(X^n(0)+\sigma_nR^n\big)}{\psi_n\big(X^n(0)\big)}
\end{equation}
being the acceptance probability. We follow \cite{Breyer2000From} in assuming that $\sigma_n=\tau/\sqrt{n}$ and $R_i$ has a finite fourth moment. 
The RWM algorithm under investigation is provided in Algorithm \ref{MHAlgorithm}. 
\begin{algorithm}[!htb]
\noindent\begin{tabular}{l}
Initialize\\
\asp\asp     Generate initial states $x_{0}^n\sim \pi_n$.\\
\asp\asp     Set $t=0$.\\
Iterate\\
\asp\asp     Generate each of $(r_1,\ldots, r_n)$ from a symmetric  PDF $\varphi(\cdot)$ with unit variance.\\
\asp\asp     Generate a candidate state $y^n=x^n(t)+\tau n^{-1/2} (r_1,\ldots, r_n)$.\\
\asp\asp     Calculate the acceptance probability $ a_n(x^n(t),y^n)=\min \left(1,{\frac {\psi_n(y^n)}{\psi_n(x^n(t))}}\right)$.\\
\asp\asp    Accept or reject:\\
\asp \asp\asp\asp    Generate a uniform random number $u\in [0,1]$.\\
\asp \asp\asp\asp    If $u\leq a_n(x^n(t),y^n)$, then set $x^n(t+1)=y^n$ else set $x^n(t+1)=x^n(t)$.\\
\asp\asp    Increment: set $t=t+1$.
\end{tabular}
\caption{(The RWM algorithm)}
\label{MHAlgorithm}
\end{algorithm}

\subsection{Existing results in \cite{Breyer2000From}}
\label{sec:existing_results}

The diffusion approach is based on the uniform convergence of generators. Specifically, according to Theorem $8$ (on page $199$) of \cite{Breyer2000From}, for $A^n$ being the discrete-time generator of the Markov chain $X^n$ and $A$ being the continuous-time generator of the limiting diffusion process $X$, under Hypotheses (H$1$-H$5$) in Theorem \ref{thm:classical_results} below, they proved
\begin{equation}
\label{eqn:generate_cvg}
\lim_{n\rightarrow \infty} |n A^nf(x)-Af(x)|=0,
\end{equation}
where $f: \mathbb{R}^{\mathbb{Z}^d}\rightarrow \mathbb{R}$ is any bounded differentiable test function that depends on at most a finite number of coordinates. Then by Corollary $8.9$ (on page $233$) of \cite{ethier2009markov}, with the additional Hypothesis (H$6$), they obtained the desired weak convergence of the time-rescaled $n$-dimensional Markov chain $X^n([tn])$ to $X$, where $[\,\cdot\,]$ stands for the integer part.

Before we display their weak convergence result (Theorem $13$, page $204$ therein) in the following theorem, we first provide the necessary definitions. 
We denote the partial derivative of a function $f$ with respect to the $k$-th component of its argument as $D_kf$ and $D_{x_k}f$ interchangeably. Analogously, we define higher-order partial derivatives, such as $D_{x_kx_j}f$ being the second-order partial derivative of $f$ with respect to the $k$-th component and the $j$-th component of its argument. The space of continuous functions $f$ is denoted as $C$, while $C^1$ is the set of such $f\in C$
that $D_k f \in C$ holds for all 
$k$. We denote $C^2$ and $C^3$ analogously.

\begin{thm}[\cite{Breyer2000From}] 
	\label{thm:classical_results}
Suppose that the following Hypotheses (H$1$)-(H$6$) hold:
\smallskip

\begin{itemize}	
\item \textbf{Hypothesis (H$1$).} Let $\mathcal{V}$ be a finite subset of $\mathbb{Z}^d$ such that $0\in \mathcal{V}$ and $v\in \mathcal{V}$ implies also that $-v\in \mathcal{V}$. For each $k\in \mathbb{Z}^d$, let $h_k:\mathbb{R}^{\mathcal{V}+k}\rightarrow \mathbb{R}$ be $C^3$, and such that $h_k \circ \oplus_l=h_{k+l}$. Assume that the family of probability measures $\Pi$ defined by \eqref{eqn:big_pi} 
is tight in the local topology and $\Pi\in G_{\oplus}(\Pi)$ the set  of  translation-invariant Gibbs distributions. 
\smallskip

\item \textbf{Hypothesis (H$2$).} The family $(\pi_n)$ is constructed from a sequence $(V_n)$ that is increasing such that $|V_n|=n$ and satisfies 
\begin{equation}
\label{eqn:boundary_ratio}
\sup_n\frac{|V_n'|}{|V_n|}<\infty,\qquad V_n'=\text{convex hull of }V_n,
\end{equation}
and as $n\rightarrow \infty$
\begin{equation}
\label{eqn:radius_cond}
\sup\{\text{radius of a sphere entirely contained in }V_n\}\rightarrow \infty.
\end{equation} 
Moreover, assume
\begin{equation}
\label{eqn:boundary_cond}
|\mathcal{V}+\partial V_n|/|V_n|<Cn^{-\alpha},
\end{equation} 
for some $\alpha>0$, where $$\partial V_n=\big\{k\in \mathbb{Z}^d:k+\mathcal{V}\nsubseteq V_n\big\}.$$

\item \textbf{Hypothesis (H$3$).} For every $m\in  \mathbb{Z}^d$, the second-order and third-order derivatives of $h_m$ are bounded:
$$\|D_{x_ix_j}h_k\|_{\infty}+\|D_{x_lx_mx_n}h_p\|_{\infty}<\infty,\qquad i,j,k,l,m,n,p\in  \mathbb{Z}^d.$$

\item \textbf{Hypothesis (H$4$).} The Gibbs measure $\pi\in G_{\oplus}(\Pi)$ satisfies, for some $\delta>1$,
$$\int |x_k|^{2\delta}\pi(dx)<\infty,\qquad k\in  \mathbb{Z}^d.$$

\item \textbf{Hypothesis (H$5$).} For each $k\in  \mathbb{Z}^d$, the functions $h_k(x)$ satisfies the Lipschitz and growth conditions (with Euclidean norm)
\begin{align*}
\max_{v\in \mathcal{V}}\|D_{x_v}h_k(x)-D_{x_v}h_k(y)\|_{\mathbb{R}^\mathcal{V}}\leq & C\|x-y\|_{\mathbb{R}^\mathcal{V}},\quad x,y\in \mathbb{R}^\mathcal{V},\\
\max_{v\in \mathcal{V}}\|D_{x_v}h_k(x)\|_{\mathbb{R}^\mathcal{V}}\leq & C(1+\|x\|_{\mathbb{R}^\mathcal{V}}),\quad x\in \mathbb{R}^\mathcal{V}.
\end{align*}

\item \textbf{Hypothesis (H$6$).} 
Define the strong mixing coefficient $\alpha_{\pi}(\mathcal{F}_{U},\mathcal{F}_{W})$ of $\pi$ for any two subsets $U$ and $W$ of $\mathbb{Z}^d$ as
$$\alpha_{\pi}(\mathcal{F}_{U},\mathcal{F}_{W})=\sup\Big\{\big|\pi(A\cap B)-\pi(A)\pi(B)\big|:A\subset \mathcal{F}_{U},\, B\subset \mathcal{F}_{W}\Big\},$$
and define the distance between sets as
$$\operatorname{dist}(A,B)=\min \left\{\max_{i\leq d}|a_i-b_i|:a\in A,\, b\in B\right\}.$$
Suppose there exists $\epsilon>0$ such that
$$\sum_{r=1}^{\infty}(r+1)^{3d-1}|\alpha_{\pi}(r)|^{\epsilon/(4+\epsilon)}<\infty,$$
where $$\alpha_{\pi}(r)=\sup\Big\{\alpha_{\pi}(\mathcal{F}_{A+\mathcal{V}},\mathcal{F}_{B+\mathcal{V}}):\operatorname{dist}(A,B)\geq r,\, |A|=|B|=2\Big\}.$$
\end{itemize}	
Then as $n\rightarrow \infty$, the expected acceptance rate of the algorithm
	\begin{equation*}
	\mathbb{E}[a_n(x^n, x^n+\tau n^{-1/2}R^n)]\rightarrow c(\tau), \quad \pi-a.e.\; x,
	\end{equation*}
where $c(\tau)$ is given in equation $(25)$ (on page $196$) of \cite{Breyer2000From} as  
\begin{equation}
\label{eqn:v_def}
c(\tau)=\mathbb{E}\Big(1\wedge \exp\left(-\tau s(\pi)Z-(1/2)\tau^2 s^2(\pi)\right)\Big)
\end{equation}
with $s(\pi)$ given in equation $(21)$ (on page $192$) of \cite{Breyer2000From} as
\begin{equation}
\label{eqn:s_pi_def}
 s(\pi)=\bigg(\int \big(D_0H_\mathcal{V}(x)\big)^2\pi(dx)\bigg)^{1/2}<\infty,
  \end{equation}
and $Z$ being a standard normal random variable.
Furthermore, in the topology of local convergence as $n\rightarrow \infty$
$$X^n([tn])\Rightarrow X(t) \quad\quad a.e.\; \pi,$$
where $X$ is the infinite-dimensional Langevin diffusion solving 
\begin{equation}
\label{eqn:limit_SDE}
X(t)=X_0+\int_{0}^{t}\tau \sqrt{c(\tau)}dB(s)-\frac{1}{2}\int_{0}^{t}\tau^2c(\tau)\nabla H(X(s))ds,
\end{equation}
where $X_0$ is distributed according to $\pi$ and $B(s)$ is an infinite-dimensional standard Brownian motion. 
\end{thm}

The following theorem provides the theoretical ergodicity support for both the diffusion approach and the Dirichlet form approach; see, Theorem $3.7$ on page $138$ of \cite{nguyen1979ergodic} for the setting of the general set $V$, whose special case as $\mathbb{Z}^d$ is covered in 
``Standard Facts'' (on page  $188$) of \cite{Breyer2000From}.
\begin{thm}[Ergodic theorem, \cite{nguyen1979ergodic}]
\label{thm:ergodic_spatial}
For $\pi\in G_{\oplus}(\Pi)$ being the only ergodic measure with respect to the group
of translations $(\oplus_k:k\in  V)$, and for any $f\in L^p(\mathbb{R}^{V};\pi)$ with $1\leq p<\infty$, let $(V_n)$ be an increasing sequence of finite subsets of $V$ such that equations \eqref{eqn:boundary_ratio} and \eqref{eqn:radius_cond} hold. Then
$$\lim_{n\rightarrow \infty}\frac{1}{|V_n|}\sum_{k\in V_n}f\circ \oplus_k=\<\pi,f\>, \qquad \pi\; a.e.\, \text{ and in } L^p(d\pi).$$
\end{thm}

\subsection{Optimal scaling}
\label{sec:Optimal_scaling}

\cite{Breyer2000From} found that there is an optimal way to scale the variance of
the proposal distribution of the RWM algorithm in order to maximise the speed of convergence of the algorithm. This
turned out to involve scaling the variance of the proposal as the reciprocal of dimension. Their findings, established within the framework of weak convergence, demonstrated that the algorithm behaves akin to an infinite-dimensional diffusion process in high dimensions. 
According to the weak convergence result in Theorem \ref{thm:classical_results}, the term $\tau \sqrt{c(\tau)}$ in equation \eqref{eqn:limit_SDE} quantifies the rate at which the limiting diffusion process evolves. Therefore, the exploration occurs at the fastest possible rate precisely when $\tau \sqrt{c(\tau)}$ is maximized.
That is, $\tau \sqrt{c(\tau)}$
 as a function of $\tau$ determines the speed of convergence of the algorithm. 
 
 Recall from equation \eqref{eqn:v_def} that we have
$$c(\tau)=\mathbb{E}\Big(1\wedge \exp\big(-\tau s(\pi)Z-(1/2)\tau^2 s^2(\pi)\big)\Big)$$
where $Z$ is a standard normal random variable. Note that by a standard calculation, for $M\sim \mathcal{N}(\sigma^2/2,\sigma^2)$, one has 
$$\mathbb{E}(1\wedge \exp(-M))=2\Phi\left(-\sigma/2\right),$$
where $\Phi$ is the cumulative distribution function of $Z$; 
this fact is widely used in optimal scaling literature (e.g. page $198$ of \cite{Breyer2000From}).
Therefore, we can rewrite 
\begin{equation}
\label{eqn:v_form}
c(\tau)=2\Phi\left(-\frac{\tau}{2}s(\pi)\right),
\end{equation}
and then 
\begin{align*}
\tau^*=\arg\max_{\tau\in\mathbb{R}^+} \{\tau^2 c(\tau)\}=\arg\max_{\tau\in\mathbb{R}^+} c(\tau)&=\arg\max_{\tau\in\mathbb{R}^+} \left\{2\Phi\left(-\frac{\tau}{2}s(\pi)\right)\right\}\nonumber\\
&\approx 2.38 /s(\pi).
\end{align*}
Plugging $\tau^*$ into equation \eqref{eqn:v_form}, we have 
$$c(\tau^*)=2\Phi\left(-\frac{\tau^*}{2}s(\pi)\right)=2\Phi(-2.38/2)\approx 0.234.$$
That is, 
the actual optimal scaling can be characterized in
terms of the overall acceptance rate of the algorithm whose maximum is $0.234$. 
%

\section{Our main results}
\label{sec:mainresults}
As we introduce Mosco convergence of Dirichlet forms on changing Hilbert spaces to MCMC optimal scaling, in Section \ref{sec:Preliminaries} we provide necessary definitions and important related results. In Section \ref{sec:Dirichlet_forms}, we identify the associated Dirichlet forms. In Section \ref{sec:Mosco_convergence}, we establish Mosco convergence of Dirichlet forms and provide its significance.

\subsection{Preliminaries}\label{sec:Preliminaries}
In contrast to \cite{zanella2017dirichlet}, which focused on the convergence of Dirichlet forms within the same Hilbert space, our Dirichlet form approach works on changing Hilbert spaces. Specifically, we work on a sequence of Hilbert spaces $\{\mathcal{H}_n\}_{n\in \mathbb{N}}$ with  $\mathcal{H}_n= L^2(\mathbb{R}^{V_n};\pi_n)$ converging to another Hilbert space $\mathcal{H}= L^2(\mathbb{R}^V;\pi)$ in the sense of Definition \ref{def:Hiblertspaces_cvg} \citep{kuwae2003convergence}.
 Let us first clarify the inner product and its related norm on $\mathcal{H}$ while the corresponding definitions on $\mathcal{H}_n$ are analogous. 
For any $f$ and $g$ in $\mathcal{H}$, denote the usual $L^2$ inner product by 
$$(f,g)_{\mathcal{H}}=\int f(x)g(x)\pi(dx)$$
and the related norm by
$$\|f\|_{\mathcal{H}}=(f,f)_{\mathcal{H}}=\int f^2(x)\pi(dx).$$
Now, we provide the rigorous definition of convergence of Hilbert spaces.
\begin{definition}[\cite{kuwae2003convergence}]
	\label{def:Hiblertspaces_cvg}
	A sequence of Hilbert spaces $ \{\mathcal{H}_n\}$ converges to a Hilbert space $\mathcal{H}$, if there exists a dense subspace $\widetilde{D}\subset \mathcal{H}$ and a sequence of operators 
	$$\Gamma_n : \widetilde{D}\rightarrow \mathcal{H}_n$$
	with the following property:
	\begin{equation}\label{EQBNPN}
	\lim_{n\rightarrow\infty}\|\Gamma_n f\|_{\mathcal{H}_n}=\|f\|_\mathcal{H}
	\end{equation}
	for every $f\in \widetilde{D}$.
\end{definition}
Here, $\Gamma_n$ is asymptotically close to a unitary operator, however it is not
necessarily injective even for sufficiently large $n$ (page $611$, \cite{kuwae2003convergence}). We note that the usual definitions of strong and weak convergence of elements in Hilbert spaces are for the situation that, all these elements as well as their limits are in the same Hilbert space. Hence, it is necessary to properly define different types of convergence of elements in changing Hilbert spaces. 
\begin{definition}[\cite{kuwae2003convergence}]
	\label{def:convergences}
\hfill
	\begin{itemize}
		\item[(1)] (Strong convergence) We say that a sequence of vectors $\{f_n\}$ with $f_n\in \mathcal{H}_n$ converges to a vector $f\in \mathcal{H}$ strongly as $n$ goes to infinity, if there exists a sequence $\{\widetilde{f}_m\}\subset \widetilde{D}$ such that
		$$
		\|\widetilde{f}_m-f\|_\mathcal{H}\rightarrow 0\quad\text{and}\quad \lim_m\limsup_n \|\Gamma_n\widetilde{f}_m-f_n\|_{\mathcal{H}_n}=0.
		$$
		\item[(2)] (Weak convergence) We say that a sequence of vectors $\{f_n\}$ with $f_n\in \mathcal{H}_n$ converges to a vector $f\in \mathcal{H}$ weakly as $n$ goes to infinity, if
		$$(f_n,g_n)_{\mathcal{H}_n}\rightarrow (f,g)_{\mathcal{H}}$$
		 for every sequence $\{g_n\}$ with $g_n\in \mathcal{H}_n$ converging to $g\in \mathcal{H}$ strongly.
		\smallskip
		
		\item[(3)] (Convergence of operators) Given a sequence of bounded linear operators $B_n$ on $\mathcal{H}_n$, we say $B_n$ strongly converges to a bounded linear operator $B$ on $\mathcal{H}$ as $n$ goes to infinity, if  $B_nf_n$ converges to $Bf$ strongly for every sequence $f_n\in \mathcal{H}_n$ converging to $f\in \mathcal{H}$ strongly.  
	\end{itemize}
\end{definition}


The lemma below states some very useful results in \cite{Kolesnikov2005Convergence}, regarding the above convergences.
\begin{lemma}[\cite{Kolesnikov2005Convergence}]
\label{useful_cvg_equi}
Let $\{f_n\}$ be a sequence with $f_n\in \mathcal{H}_n$ and let $f\in \mathcal{H}$. Then the following holds:
	\begin{enumerate}
		\item If $f_n\rightarrow f$ strongly, then $\|f_n\|_{\mathcal{H}_n}\rightarrow \|f\|_{\mathcal{H}}$.
		\smallskip
		
		\item For every $f\in \mathcal{H}$ there exists a sequence $\{f_n\}$ such that $f_n\rightarrow f$ strongly.
		\smallskip
		
		\item If the sequence of norms $\|f_n\|_{\mathcal{H}_n}$ is bounded, there exists a weakly convergent subsequence of $\{f_n\}$.
		\smallskip

		\item If $\{f_n\}$ is a sequence that weakly converges to $f$, then 
		$$\sup_n\|f_n\|_{\mathcal{H}_n}<\infty\quad\text{and}\quad 
		\|f\|_{\mathcal{H}}\leq \liminf_{n\rightarrow \infty} \|f_n\|_{\mathcal{H}_n}.$$
		Moreover, $f_n\rightarrow f$ strongly if and only if $\|f\|_{\mathcal{H}}=\lim_{n\rightarrow \infty} \|f_n\|_{\mathcal{H}_n}$.
		\smallskip

		\item The sequence $\{f_n\}$ tends to $f$ strongly if and only if $$(f_n,g_n)_{\mathcal{H}_n}\rightarrow (f,g)_{\mathcal{H}}$$
		 for every sequence $\{g_n\}$ with $g_n\in \mathcal{H}_n$ converging to $g\in \mathcal{H}$ weakly.
		\smallskip
		
		\item Assume that $\int_K (\psi_n(x)-\psi(x))^2dx\rightarrow 0$ for any compact set $K$. Then $f_n$ converges to $f$ strongly (resp. weakly) in $\mathcal{H}$ if and only if $f_n\psi_n$ converges to $f\psi$ strongly (resp. weakly) in $L^2$.
	\end{enumerate}
\end{lemma}


\subsection{Dirichlet forms}
\label{sec:Dirichlet_forms}
%

A Dirichlet form is a bilinear form, also known as a quadratic form, satisfying certain regularity conditions. We firstly recall the definition of symmetric bilinear forms here.
\begin{definition}
For a function $\sE: D(\sE)\times D(\sE)\rightarrow \mathbb{R}$ where its domain $D(\sE)$ is some subspace of $\mathcal{H}$, we say that it is a symmetric bilinear form over $\mathbb{R}$, if the following three conditions hold: for any $u,v,w\in D(\sE)$ and $\lambda\in \mathbb{R}$,
\begin{enumerate}
\item $\sE(u,v)=\sE(v,u)$;
\smallskip

\item $\sE(\lambda u,v)=\lambda\sE(u,v)$;
\smallskip

\item $\sE(u+v,w)=\sE(u,w)+\sE(v,w)$.
\end{enumerate}
\end{definition}
Here, we provide a concise definition for symmetric Dirichlet forms, while the general definition covering the non-symmetric case can be seen in Definition $4.5$ (on page $34$) of  \cite{ma2012introduction}.
\begin{definition} 
\label{def:D_form}
We say that a symmetric bilinear form $\sE: D(\sE)\times D(\sE)\rightarrow \mathbb{R}$ is a (symmetric) Dirichlet form on $\mathcal{H}$, if the following three conditions hold: for all $u,v\in D(\sE)$,
\begin{enumerate}
\item $\sE$ is non-negative definite, i.e., $\sE(u)=\sE(u,u)\geq 0$;
\smallskip

\item $\sE$ is closed on $\mathcal{H}$, i.e., $D(\sE)$ is complete with respect to the norm $\sE_1^{1/2}$ where $\sE_1(u,v)=\sE(u,v)+(u,v)_{\mathcal{H}}$;
\smallskip

\item one has $u^{\ast}=\inf(\sup(u,0),1)\in D(\sE)$ and $\sE(u^{\ast})\leq \sE(u)$.
\end{enumerate}
\end{definition}

The (symmetric) Dirichlet form $\sE^n$ associated with  the time-rescaled RWM  chain $X^n([tn])$, takes the following form:
\begin{align}\label{En}
\sE^n(u)=&\frac{n}{2}\int \pi_n(dx)P_n(x,dy)[u(y)-u(x)]^2\nonumber\\
=&\dfrac{n}{2}\mathbb{E}[u(X^n(1))-u(X^n(0))]^2,
\end{align}
for any $u\in D(\sE^n):=\{f\in \mathcal{H}_n:\sE^n(f)<\infty\}$, where $P_n$ is the transition probability; 
see page $11$ of \cite{chen2006eigenvalues} for reference on Dirichlet forms associated with discrete-time Markov chains. 
We commit a slight abuse of notation by identifying $\sE^n$ on $\mathcal{H}_n$ through the following extension:
$$\mathcal{E}^n(\cdot): \mathcal{H}_n\rightarrow \overline{\mathbb{R}}:=\mathbb{R}\cup\{\infty\},\quad\quad \sE^n(u) = \left\{ \begin{array}{lcl}
\sE^n(u), && \mbox{for}
\; u\in D(\sE^n), \\ 
\infty, && \mbox{for}\; u \in \mathcal{H}_n\backslash D(\sE^n).
\end{array}\right.$$

As we conduct asymptotic analysis by sending $n$ to infinity, one crucial step is to identify the limiting form of $\sE^n$. Given that 
the acceptance probability function $a_n$
plays a crucial role in generating $X^n(1)$ from $X^n(0)$ by the recursive equation \eqref{eqn:X1X0} of the RWM algorithm. 
We firstly explore the limiting behavior of $a_n$ in the following lemma, whose proof is postponed to Section \ref{sec:Proof_a_convergence}.
\begin{thm}\label{thm:a_convergence}
	Suppose that Hypothesis (H$2$) in Theorem \ref{thm:classical_results} and the following Assumptions (A$1$) and (A$2$) hold:
	\begin{itemize}	
\item \textbf{Assumption (A$1$).} Let $\mathcal{V}$ be a finite additive subset of $V$ in graph distance. For each $k\in V$, let $h_k:\mathbb{R}^{\mathcal{V}+k}\rightarrow \mathbb{R}$ satisfy $h_k \circ \oplus_l=h_{k+l}$. We assume that the family of probability measures $\Pi$ defined by \eqref{eqn:big_pi} 
is tight in the local topology and $\Pi\in G_{\oplus}(\Pi)$. 
\smallskip


\item \textbf{Assumption (A$2$).} 
For every $m\in  V$, the second-order derivatives of $h_m$ are bounded:
$$\|D_{x_ix_j}h_m\|_{\infty}<\infty,\qquad i,j\in  V.$$

\end{itemize}	
    Then we have that for $N$ being a fixed positive integer, as $n\rightarrow\infty$,     
	$$\mathbb{E}\left[a_n(X^n, X^n+\tau n^{-1/2}R^n)\mid X^n, R^N\right]\longrightarrow c(\tau), \quad \pi-a.e.$$
	where $c(\tau)$ is given in \eqref{eqn:v_def}.
\end{thm}

Although the infinite-dimensional space $\mathbb{R}^{V}$ is not a Banach space because its topology cannot be derived from any norm,
it is a locally convex Polish
space (i.e., a locally convex, separable, and completely metrizable topological space) with respect to the product topology since $V$ is countable (page $479$, \cite{fritz1982stationary}). 
To explore Dirichlet forms in the infinite-dimensional space, we first define a linear space of functions by
\begin{align}
\label{eqn:mathscr_F}
\mathscr{F}C_{b}^\infty=\bigg\{f(x_{1},\ldots, x_{m})\; \Big|\; &x\in \mathbb{R}^{V},\;  m\in \mathbb{N},\; f\in C_b^\infty(\mathbb{R}^m)\bigg\},
\end{align}
where  $C_b^\infty(\mathbb{R}^m)$ denotes the set of all infinitely differentiable (real-valued) functions on $\mathbb{R}^m$ with bounded partial derivatives. By the Hahn-Banach theorem, $\mathscr{F}C_b^\infty$ separates the points of $\mathbb{R}^V$; hence, since $\mathcal{B}(\mathbb{R}^{V})=\sigma(\mathbb{R}^{V})$,
$$\mathscr{F}C_{b}^\infty\text{ is dense in } L^2(\mathbb{R}^V;\pi),$$
by monotone class arguments (page $54$ of \cite{ma2012introduction} and the references therein). 

Next, we need to properly define the partial derivatives in the infinite-dimensional space. Define for $u\in \mathscr{F}C_b^\infty$ and $e_k, z\in \mathbb{R}^{V}$,
$$(\partial u/\partial e_k)(z):=(d/ds)u(z+s e_k)|_{s=0},$$ 
where $e_k=\{\delta_{k,j}\}_{j\in \mathbb{R}^{V}}$ with $\delta$ standing for the Kronecker delta.
Then. by Theorem \ref{thm:a_convergence}, we identify a partial Dirichlet form on $\mathcal{H}$ as 
$$\sE_k(u)=\frac{\tau^2 c(\tau)}{2}\int |\partial u/\partial e_k|^2d\pi.$$
Hence, $(\sE_k, \mathscr{F}C_b^\infty)$ is a 
positive definite symmetric bilinear form on $\mathcal{H}$. By the regularity conditions imposed in Theorem \ref{thm:a_convergence}, and by Theorem $1.2$ (on page $126$) of \cite{albeverio1990partial}, $(\sE_k, D(\sE_k))$
where $D(\sE_k)=\{f\in \mathcal{H}:\sE_k(f)<\infty\}$,
 is a closed extension of $(\sE_k, \mathscr{F}C_b^\infty)$, according to the sense of closedness in Definition \ref{def:D_form}.

Next, recalling that $V$ is countable, we define the linear space
\begin{align}
\label{def:S}
S=\Bigg\{u\in \bigcap_{k\in V}D(\sE_k)\;\Bigg|\; &\text{there exists a }\mathcal{B}(\mathbb{R}^{V})/\mathcal{B}(\mathcal{H})-\text{measurable function }\nonumber\\
&\nabla u: \mathbb{R}^V\rightarrow \mathcal{H}\text{ such that, for each }k\in V,\nonumber\\
&(\nabla u(z),k)_{\mathcal{H}}=(\partial u/\partial e_k)(z)\text{ for }\pi-\text{a.e. } z\in \mathbb{R}^V\nonumber\\
& \text{and } \int  |\nabla u|^2d\pi<\infty
 \Bigg\}, 
\end{align}
and then for $u\in S$ we define
\begin{align}\label{E}
\sE(u)=\sum_{k\in V}\sE_k(u)=\frac{\tau^2 c(\tau)}{2}\int |\nabla u|^2 d\pi.
\end{align}
By Theorem $1.6$ (on page $128$) of \cite{albeverio1990partial} and the regularity conditions imposed in Theorem \ref{thm:a_convergence}, we know that $(\sE,  S)$ is a Dirichlet form. At last, we commit a slight abuse of notation by identifying $\sE$ on $\mathcal{H}$ through the following extension:
$$\mathcal{E}(\cdot): \mathcal{H}\rightarrow \overline{\mathbb{R}},\quad\quad\sE(u) = \left\{ \begin{array}{lcl}
\sE(u), && \mbox{for}
\; u\in S, \\ 
\infty, && \mbox{for}\; u \in \mathcal{H}\backslash S.
\end{array}\right.$$

\subsection{Comparison and examples}
\label{sec:Comparison_examples}

	Assumptions (A$1$) and (A$2$) in Theorem \ref{thm:a_convergence} are all the assumptions needed with the Dirichlet form approach, we compare these with Hypotheses (H$1$)-(H$6$) in Theorem \ref{thm:classical_results} in the following:
	\begin{itemize}	 
		\item Assumption (A$1$) broadens the scope of models by relaxing constraints on regularity conditions and the graph structure.
		\begin{itemize}
			\item Assumption (A$1$) eliminates the necessity for $h_k$ to be $C^3$ for every $k\in \mathbb{Z}^d$.
			
			\item Recall that $\mathbb{Z}^d$ is a special case of the graph $G$ here. In Assumption (A$1$), $\mathcal{V}$ as a finite additive subset of $V$ in graph distance extends the strict symmetric setting in Hypothesis (H$1$) by including also the asymmetric setting. In other words, here the additive property that $0\in \mathcal{V}$ and $v\in \mathcal{V}$ implies also that $-v\in \mathcal{V}$, says that $|0-v|=|0+v|$ in the graph distance. Note that, in the special $\mathbb{Z}^d$ setting, it is equivalent to the Euclidean distance-based additive property. One illustration can be seen in Figure \ref{fig:additive_illustration}. 
			
			\item Spatial homogeneity $h_k:\mathbb{R}^{\mathcal{V}+k}\rightarrow \mathbb{R}$ is commonly employed in spatial models; see, for example, page 4 of \cite{JSSv101i08} and page 18 of \cite{ning2024vt}, as well as the references therein.
		\end{itemize}
		
		
		\item Assumption (A$2$) relaxes the requirement for bounded third-order derivatives specified in Hypothesis (H$3$). 
		\smallskip
		
		\item Hypotheses (H$4$)-(H$6$) are removed. No example satisfying Hypothesis (H6) was provided in \cite{Breyer2000From}. Due to its complexity, verifying instances poses a significant challenge, and as a result, we are unable to locate such examples or counterexamples. 
	\end{itemize}



We provide a general example that does not satisfy the bounded third-order derivative constraint in (H3). 
\begin{example}
	\label{example:no_3rd}
	Consider the density of the Gibbs distribution 
	$$
	f(x)=Z_P^{-1} \times \mathrm{e}^{-H(x)}.
	$$
	The energy function could be defined as the sum of clique potentials $V_{\card(c)}(x)$ over all possible cliques $\mathcal{C}$, i.e.,
	$$
	H(x)=\sum_{c \in \mathcal{C}} V_{\card(c)}(x),
	$$
	where $V_{\card(c)}(x)$ depends on the local configuration within the clique $c$, with $\card(c)$ being the cardinality of $c$. Specificially,
	\begin{align}
		\label{eqn:generalU}
		H(x)=\sum_{\{i\} \in \mathcal{C}_1} V_1\left(x_i\right)+\sum_{\left\{i,j\right\} \in \mathcal{C}_2} V_2(x_i, x_j)+\sum_{\left\{i,j,k\right\} \in \mathcal{C}_3} V_3(x_i, x_j, x_k)+\cdots.
	\end{align}
	Consider a compact domain $x\in [-R,R]$ for any $R>0$.
	For example, if $V_1(x_i) \propto x_i^{\alpha}$ for any $\alpha\in (2,3)$, then we have the second-order derivative
	$$V_1''(x_i)=\alpha (\alpha-1)x_i^{\alpha-2}$$
	being bounded in the domain, 
	but  the third-order derivative $$V_1'''(x_i)=\alpha (\alpha-1)(\alpha-2)x_i^{\alpha-3}$$
	which is infinite at $x_i=0$. Or if $V_2(x_i, x_j) \propto x_i^{\alpha-1}\eta(x_j)$ for any differentiable function $\eta$,
	then we have the second-order derivative
	$$\partial^2 V_2(x_i, x_j)/\partial x_i\partial x_j =(\alpha-1) x_i^{\alpha-2}\eta'(x_j)$$
	being bounded in the domain, 
	but the third-order derivative 
	$$\partial^3 V_2(x_i, x_j)/\partial x_i^2\partial x_j =(\alpha-1) (\alpha-2)x_i^{\alpha-3}\eta'(x_j)$$
	which is infinite at $x_i=0$. A concrete practical model is the SABR model (Stochastic Alpha, Beta, Rho) introduced by \cite{hagan2002managing}, along with its diverse multivariate variants (e.g., \cite{ferreiro2014sabr} and the references therein).
	Clearly, for any function $V_i$, as long as one of its partial derivatives up to the third order being unbounded, it fails to meet the requirement of having bounded third-order derivatives. Consequently, this restriction greatly limits the possibilities for graphical modeling.
\end{example}

In the following example, we present one practical example that does not satisfy Hypothesis (H$4$).
\begin{example}
	\label{example3}
	The $q$-Gaussian distribution is a generalization of the Gaussian distribution and finds applications across various domains such as statistical mechanics, geology, anatomy, astronomy, economics, finance, machine learning, and more; see, for instance, \cite{douglas2006tunable,thistleton2007generalized,furuichi2009maximum,vignat2009detection,bercher2012generalized}. The standard $q$-Gaussian distribution has the following PDF:
	$$f(x)={{\sqrt {\beta }} \over C_{q}}\Big[1+(1-q)(-\beta x^{2})\Big]_{+}^{1 \over 1-q}, \qquad q<3,$$
	where the positive parameter $\beta=1/(3-q)$ and $C_{q}$ is the normalization factor.
	For 
	$1\leq q<3$, its support is $x\in (-\infty ;+\infty )$, and for  $q<1$, that is 
	$x\in [-1 /{\sqrt {\beta (1-q)}},1 /{\sqrt {\beta (1-q)}}]$.
	Clearly, when $q=1$ it is the standard Gaussian distribution. 
	When $q\geq 3/2$, its third moment is undefined, hence failing to satisfy Hypothesis (H$4$).
\end{example}

In the following example, we present one practical example that does not satisfy Hypothesis (H$5$).
\begin{example}
	\label{example4}
	The generalized normal distribution (also known as the exponential power distribution) is a family of continuous probability distributions that generalizes the normal distribution by introducing an additional shape parameter \(\beta\), which controls the kurtosis (tailedness) of the distribution. 
	The PDF of a generalized normal distribution is given by
	\[
	f(x \mid \mu, \alpha, \beta) = \frac{\beta}{2\alpha \Gamma\left(1/\beta\right)} \exp\left( - \left(|x - \mu|/\alpha\right)^\beta \right),
	\]	
	where \(\mu\) is the location parameter,
	\(\alpha > 0\) is the scale parameter,
	\(\beta > 0\) is the shape parameter, and 
	\(\Gamma(\cdot)\) is the Gamma function. 
	The shape parameter \(\beta\) controls the tail behavior:
	When \(\beta = 2\), the distribution corresponds to the normal distribution. When \(\beta > 2\), the tails of the distribution are lighter than those of the normal distribution; see, for example, \cite{dytso2018analytical} and \cite{banerjee2013underwater}, along with the references cited therein, for a discussion of its properties and applications.
	
	The first-order derivative of \(\log(f(x \mid \mu, \alpha, \beta))\) is given by
	\[
	\frac{d}{dx} \log(f(x\mid \mu, \alpha, \beta)) = - \beta \frac{|x - \mu|^{\beta - 1}}{\alpha^\beta} \operatorname{sign}(x - \mu),
	\] 	
	where \(\operatorname{sign}(x - \mu)\) is the sign function, which takes the value \(1\) for \(x > \mu\), \(-1\) for \(x < \mu\), and \(0\) at \(x = \mu\). When \(\beta > 2\), $\frac{d}{dx} \log(f(x\mid \mu, \alpha, \beta))$ does not meet the growth condition required by Hypothesis (H$5$). 
A numerical demonstration of how to apply the optimal scaling result derived in the next section, specifically for the case \(\mu = 0\), \(\alpha = 1\), and \(\beta =3\), is presented in Appendix \ref{sec:Numerical}.
\end{example}

\subsection{Mosco convergence of Dirichlet forms}\label{sec:Mosco_convergence}
The definition of Mosco convergence, is originally due to Mosco (Definition $2.1.1$ in \cite{mosco1992composite}, page $375$), on a fixed
Hilbert space.  Later, \cite{kuwae2003convergence} (Definition $2.11$, page $626$) provided the definition of Mosco convergence on changing Hilbert spaces.


\begin{definition}[\cite{kuwae2003convergence}]
	\label{def:Mosco_converges}
	 We say that a sequence of bilinear forms $\{\mathcal{E}^n: \mathcal{H}_n\rightarrow \overline{\mathbb{R}}\}$ Mosco converges to a bilinear form $\mathcal{E}: \mathcal{H}\rightarrow \overline{\mathbb{R}}$, if the following conditions hold:
	\begin{itemize}
		\item[(M$1$)] If a sequence $\{f_n\}$ with $f_n\in \mathcal{H}_n$ converges to $f\in \mathcal{H}$ weakly, then
		$$\mathcal{E}(f)\leq\liminf_{n\rightarrow \infty}\mathcal{E}^n(f_n).$$
		\item[(M$2$)] For every $f\in \mathcal{H}$, there exists a sequence $\{f_n\}$ with $f_n\in \mathcal{H}_n$ converging to $f$ strongly such that 
		$$\mathcal{E}(f)=\lim_{n\rightarrow \infty}\mathcal{E}^n(f_n).$$
	\end{itemize}
\end{definition}

The following is the main result of this paper.
\begin{thm}
	\label{thm:Mosco_cvg}
Under the assumptions imposed in Theorem \ref{thm:a_convergence}, the sequence of Dirichlet forms $\{\mathcal{E}^n: \mathcal{H}_n\rightarrow \overline{\mathbb{R}}\}$  (defined in \eqref{En}) Mosco converges to the Dirichlet form $\mathcal{E}: \mathcal{H}\rightarrow \overline{\mathbb{R}}$ (defined in \eqref{E}).
\end{thm}
\begin{proof}
According to Definition \ref{def:Mosco_converges}, completing the proof requires establishing both (M$1$) and (M$2$), which we accomplish in Sections \ref{sec:ProofsM1} and \ref{sec:ProofsM2}, respectively.
\end{proof}

We now prepare to provide the significance of Mosco convergence (cf, Theorem $2.4$ in 
\cite{kuwae2003convergence}, page $627$). We associate the Dirichlet form $\mathcal{E}$ with a non-negative self-adjoint operator $-A$, such that $\mathcal{D}(\sqrt{-A})=\mathcal{D}(\mathcal{E})$ and $\sE(u,v)=(-Au, v)$ for $u,v$ in $\mathcal{D}(\mathcal{E})$.
$A$ is referred to as the infinitesimal generator of the semigroup $T_t$, which satisfies the following equation:
\begin{align}
	\label{eqn:Loperator}
Au=\lim_{t\downarrow 0}\frac{1}{t}(T_tu-u)
\end{align}
for $u\in \mathcal{D}(A)$; see, Definition $1.8$ on page $10$ of \cite{ma2012introduction}.  Let $(T^n_t)_{t\geq 0}$ (resp. $(T_t)_{t\geq 0}$) be the semigroup of $\mathcal{E}^n$ (resp. $\mathcal{E}$). 
As outlined in Section \ref{sec:existing_results}, the strong convergence of ${T^n_{t}}$ forms the cornerstone of the diffusion approach. By equation $(26)$ (on page $198$) of \cite{Breyer2000From}, the  infinitesimal generator of $X$ is given by
$$Af(x)=\tau^2c(\tau) \left( -\frac{1}{2}\big\langle \nabla H(x), \nabla f(x) \big \rangle + \frac{1}{2}\Delta f(x) \right),$$
where $c(\tau)$ is given in \eqref{eqn:v_form}.
Then by $$	\sE(u)=\lim_{t\rightarrow 0} \frac{1}{2}\frac{\big\langle(I-T_t)u, (I-T_t)u \big\rangle}{t},
$$
one could verify that $\mathcal{E}$ and $X$ are identical; we refer interested readers to Definition $1.6$ on page $9$, Definition $1.8$ on page $10$, Theorem $2.15$ on page $22$, equation $(1.8)$ on page $91$, of \cite{ma2012introduction}, for reference on Dirichlet forms associated with continuous-time Markov chains. 
\begin{corollary}
	\label{thm:corollary}
	The Mosco convergence of $\mathcal{E}^n$ to $\mathcal{E}$ is equivalent to the strong convergence of $T^n_{t}$ to $T_t$ for every $t>0$. The limiting Dirichlet form $\mathcal{E}$ and the limiting diffusion $X$ in \eqref{eqn:limit_SDE} are identical. Subsequently, the standard optimal scaling procedure outlined in Section \ref{sec:Optimal_scaling} follows.
\end{corollary}

\section{Technical proofs}
\label{sec:Technical_proofs}
In this section, we present the proof of Theorem \ref{thm:a_convergence} in Section \ref{sec:Proof_a_convergence} and conclude the proof of Theorem \ref{thm:Mosco_cvg} by establishing Conditions (M$1$) and (M$2$) in Sections \ref{sec:ProofsM1} and \ref{sec:ProofsM2}, respectively. Throughout the proofs, we employ the following notations: $X^N=(X_1,X_2,\ldots,X_N)$ for a fixed positive integer $N$, $X^n=(X_1,X_2,\ldots,X_n)$,  $X^{n-N}=(X_{N+1},X_{N+2},\ldots,X_n)$ when $n> N$ which is certainly true as we consider sending $n$ to infinity, and $X=(X_1,X_2,\ldots)$.  Analogous notations apply to $R$. Recall that $\psi_n$ represents the density of $\pi_n$, and $\varphi$ stands for the density of $R_i$ for any $i\in \mathbb{N}$. By the independence of $\{R_i\}_i$, we denote $\varphi^n$ (resp. $\varphi^N$, $\varphi^{n-N}$) as the density of $R^n$ (resp. $R^N$, $R^{n-N}$). 
Prior to the rigorous proofs, we elaborate on the challenges encountered in this general model setting and describe the necessary innovations in our proof strategies.

\subsection{Proof of Theorem \ref{thm:a_convergence}}
\label{sec:Proof_a_convergence}

In this section, we present the proof of Theorem \ref{thm:a_convergence}. Before delving into the formal proof in Section \ref{sec:Proofa_convergence}, we outline the proof strategy in a more accessible manner in Section \ref{sec:Proofsa_cvg} to enhance understanding. Section \ref{sec:Proofsa_corollary} covers Corollary \ref{corollary} which will be utilized in proving Condition (M$1$) in the next section.

\subsubsection{Sketch of the proof of Theorem \ref{thm:a_convergence}}
\label{sec:Proofsa_cvg}
The result of Theorem \ref{thm:a_convergence} is analogous to Lemma $17$ and Corollary $18$ (on pages $4064-4067$) of  \cite{zanella2017dirichlet}, but our approach to proving it differs significantly. 
With the product-formed target distribution, their proof started with a logarithmic transformation that yielded a desired summation form (cf. equation $(17)$ on page $4064$) as follows:
$$\log\left(\prod_{i=1}^n\frac{f(x_i+ \tau n^{-1/2}R_i)}{f(x_i)}\right)=\sum_{i=1}^n \Big[\log f(x_i+\tau n^{-1/2}R_i)-\log f(x_i)\Big],$$
and ended with the following inequality:
\begin{align*}
&\Bigg|  \mathbb{E}\Bigg[ 1\wedge \frac{f(X^n+\tau n^{-1/2}R^n)}{f(X^n)} -1\wedge \frac{f(X^{n-N}+\tau n^{-1/2}R^{n-N})}{f(X^{n-N})} \Bigg | X^n, R^N \Bigg] \Bigg|\\
&\leq \Bigg|  \frac{f(X^N+\tau n^{-1/2}R^N)}{f(X^N)} \Bigg|\longrightarrow 0,
\end{align*}
as $n$ goes to infinity. However, neither of these techniques is viable with a general (dependent) target distribution. Consequently, the approach utilized in \cite{zanella2017dirichlet} is not applicable here.

We reduce the problem to establish a weak convergence result of the difference of Hamiltonian functionals (equation \eqref{eqn:cvg_Hamiltonian2}), leveraging the specific form of the acceptance probability function. The weak convergence limit is normally distributed with the mean being half of the variance. Employing the Taylor series expansion, we decompose the difference of Hamiltonian functionals into several terms: the first contributes to a limiting normal distribution with mean zero but desired variance; the second term 
$-\tau^2s^2(\tau)/2$ that is a constant, serves as the desired mean of limiting distribution; the remaining terms converge to zero in probability. Throughout the proof, we rely on the Ergodic theorem for spatial processes (Theorem \ref{thm:ergodic_spatial}) and the imposed regularities.

\subsubsection{Proof of Theorem \ref{thm:a_convergence}}
\label{sec:Proofa_convergence}
	
Throughout the proof, we condition implicitly on $X_i=x_i$ for $i\in \{1,\ldots,n\}$ and $R_j=r_j$ for $j\in \{1,\ldots,N\}$. Denote 
\begin{align}
	\label{eqn:widetildeR}
	\widetilde{R}^n=(\widetilde{R}_1,\ldots,\widetilde{R}_N,\widetilde{R}_{N+1},\ldots,\widetilde{R}_n):=(r_1,\ldots,r_N,R_{N+1},\ldots,R_n).
\end{align}
 By the acceptance probability function defined in \eqref{eqn:def_acceptance_probability}, we have 
	\begin{align*}
		a_n(x^n, x^n+\tau n^{-1/2}\widetilde{R}^n)=1\wedge\dfrac{\exp\big(-H_{V_n}\big(x^n+\tau n^{-1/2}\widetilde{R}^n\big)\big)}{\exp\big(-H_{V_n}\big(x^n\big)\big)}.
	\end{align*}
	To prove
	$$\mathbb{E}\left[a_n(X^n, X^n+\tau n^{-1/2}R^n)\mid X^n=x^n, R^N=r^N\right]\longrightarrow c(\tau), $$
	with $c(\tau)$ (defined in 
	\eqref{eqn:v_def}) given by
	$$c(\tau)=\mathbb{E}\Big(1\wedge \exp\left(-\tau s(\pi)Z-(1/2)\tau^2 s^2(\pi)\right)\Big),$$
	where $Z$ is a standard normal random variable and $s(\pi)$ is given in \eqref{eqn:s_pi_def}, it suffices to show that as $n\rightarrow\infty$, we have that 
	\begin{align}
		\label{eqn:cvg_Hamiltonian2}
		H_{V_n}(x^n+\tau/\sqrt{n} \widetilde{R}^n)-H_{V_n}(x^n)\Longrightarrow \tau s(\pi) Z+\tau^2s^2(\pi)/2.
	\end{align}
	That sufficiency holds by the monotonicity and continuity of the exponential function,
	the fact that a sequence of uniformly bounded random variables is uniformly integrable, and Theorem $3.6$ (on page $31$) of \cite{billingsley2013convergence} which states that if $\{\widehat{Y}_n\}$ are uniformly integrable and $\widehat{Y}_n \Rightarrow \widehat{Y}$ then $\widehat{Y}$ is integrable and $\mathbb{E}\widehat{Y}_n\rightarrow \mathbb{E}\widehat{Y}$.
	
Given that 
$$H_{V_n}(x^n+\tau/\sqrt{n} \widetilde{R}^n)-H_{V_n}(x^n)=\log \Big(1+\big[e^{H_{V_n}(x^n+\tau/\sqrt{n} \widetilde{R}^n)-H_{V_n}(x^n)}-1\big]\Big),$$
taking $J_n=e^{H_{V_n}}$, 
under the imposed conditions, we could 
apply Taylor's expansion and obtain
\begin{align*}
	&H_{V_n}(x^n+\tau/\sqrt{n} \widetilde{R}^n)-H_{V_n}(x^n)\\
	&=\log\Big(1+\big[J_n(x^n+\tau/\sqrt{n} \widetilde{R}^n)/J_n(x^n)-1\big]\Big)\\
	&=\big[J_n(x^n+\tau/\sqrt{n} \widetilde{R}^n)/J_n(x^n)-1\big]-\frac{1}{2}\big[J_n(x^n+\tau/\sqrt{n} \widetilde{R}^n)/J_n(x^n)-1\big]^2\\
	&\quad+\frac{1}{3}\alpha_n\big|J_n(x^n+\tau/\sqrt{n} \widetilde{R}^n)/J_n(x^n)-1\big|^3,
\end{align*}
where $|\alpha_n|\leq 1$. Organizing the above terms, we obtain
\begin{align}
	\label{eqn:split_terms}
	&H_{V_n}(x^n+\tau/\sqrt{n} \widetilde{R}^n)-H_{V_n}(x^n)\nonumber\\
	&=\frac{\tau}{\sqrt{n}}\sum_{i\in V_n} \widetilde{R}_i\frac{D_{x_i} J_n(x^n)}{J_n(x^n)}+\tau^2s^2(\pi)/2\nonumber\\
	&\quad+\Bigg\{\big[J_n(x^n+\tau/\sqrt{n} \widetilde{R}^n)/J_n(x^n)-1\big]-\frac{\tau}{\sqrt{n}}\sum_{i\in V_n} \widetilde{R}_i\frac{D_{x_i} J_n(x^n)}{J_n(x^n)}\nonumber\\
	&\quad\quad\quad-\frac{\tau^2}{2n} \sum_{i,j\in V_n} \widetilde{R}_i \widetilde{R}_j\frac{D_{x_ix_j} J_n(x^n)}{J_n(x^n)}\Bigg\}\nonumber\\
	&\quad-\frac{1}{2}\left\{\big[J_n(x^n+\tau/\sqrt{n} \widetilde{R}^n)/J_n(x^n)-1\big]^2-\frac{\tau^2}{n}\left(\sum_{i\in V_n} \widetilde{R}_i\frac{D_{x_i} J_n(x^n)}{J_n(x^n)}\right)^2\right\}\nonumber\\
	&\quad-\frac{1}{2}\left\{\frac{\tau^2}{n}\left(\sum_{i\in V_n} \widetilde{R}_i\frac{D_{x_i} J_n(x^n)}{J_n(x^n)}\right)^2-\frac{\tau^2}{n} \sum_{i,j\in V_n} \widetilde{R}_i \widetilde{R}_j\frac{D_{x_ix_j} J_n(x^n)}{J_n(x^n)}+\tau^2s^2(\pi) \right\}\nonumber\\
	&\quad+\frac{1}{3}\Big\{\alpha_n\big|J_n(x^n+\tau/\sqrt{n} \widetilde{R}^n)/J_n(x^n)-1\big|^3\Big\}\nonumber\\
	&=:\frac{\tau}{\sqrt{n}}\sum_{i\in V_n} \widetilde{R}_i\frac{D_{x_i} J_n(x^n)}{J_n(x^n)}+\tau^2s^2(\pi)/2 +\mathcal{T}_{n1}-\frac{1}{2}\mathcal{T}_{n2}-\frac{1}{2}\mathcal{T}_{n3}+\frac{1}{3}\mathcal{T}_{n4}.
\end{align}
The proof is complete upon the following results hold:
\begin{enumerate}
	\item the law $\mathcal{L}\left(\frac{\tau}{\sqrt{n}}\sum_{i\in V_n} \widetilde{R}_i\frac{D_{x_i} J_n(x^n)}{J_n(x^n)}\right)
	\rightarrow \mathcal{N}(0,\tau^2 s^2(\pi)),$
	\smallskip
	
	\item[(k)] $\mathcal{T}_{n(k-1)}$ goes to zero in probability as $n$ goes to infinity ($k=2, 3, 4, 5$),
\end{enumerate}
which will be our goals.\\

\noindent \textbf{Proof of $(1)$.} 
By the relation $J_n=e^{H_{V_n}}$, clearly
\begin{align}
	\label{eqn:derivative_equ}
	\frac{D_{x_i} J_n(x^n)}{J_n(x^n)}=D_{x_i} H_{V_n}(x^n).
\end{align}
If $k\in V_n\backslash \partial V_n$, by the homogeneity and finite range interaction imposed in Assumption (A$1$), without loss of generality suppose $\mathcal{V}=\{0, v^1, \ldots, v^m\}$, and then we have
\begin{align}
	D_{x_k}H_{V_n}(x^n)=&D_{x_0}h_k(x_k,x_{k+v^1},\ldots,x_{k+v^m})+\cdots\nonumber\\
	&+D_{x_{v^m}}h_{k-v^m}(x_{k-v^m},x_{k+v^1-v^m},\ldots,x_{k})\nonumber\\
	=&\sum_{i\in \mathcal{V}}D_{x_i}h_{k-i}(x)\nonumber\\
	=&\left(\sum_{i\in \mathcal{V}}D_{x_i}h_{-i}\right)\circ\oplus_k(x)\nonumber\\
	=&D_{x_0}H\circ\oplus_k(x)\nonumber\\
	=&D_{x_k}H(x).\label{eqn:density_finite_range_trans}
\end{align}
Then by the boundary growth condition imposed in Assumption (A$2$), we have
$$\frac{\tau}{\sqrt{n}}\sum_{k\in V_n} D_{x_k} H_{V_n}(x^n)\simeq\frac{\tau}{\sqrt{n}}\sum_{k\in V_n} D_{x_k}H(x),$$
where $a_n\simeq b_n$ means that $\lim_{n\rightarrow \infty}|a_n-b_n|=0$. 
Recall that each $\widetilde{R}_k$  for $k\in \{N+1,\ldots,n\}$ is independent to the others, and has zero mean and unit variance, where $N$ is a fixed integer.
Next, by Assumption (A$1$) and Hypothesis (H$2$) together of which allows us to apply  Theorem \ref{thm:ergodic_spatial} (the Ergodic theorem), we have that, $\pi$ a.e. $x$,
\begin{align}
	\label{eqn:lemma_ergodic}
	\frac{1}{n}\sum_{k\in V_n}|D_{x_k}H(x)|^2\rightarrow s^2(\pi).
\end{align}
Through a standard central limit theorem argument, we have
$$\mathcal{L}\left(\frac{\tau}{\sqrt{n}}\sum_{k\in V_n} D_{x_k} H_{V_n}(x^n)\widetilde{R}_k\right)
\rightarrow \mathcal{N}(0,\tau^2 s^2(\pi)),$$
which completes the proof of $(1)$.
\medskip

\noindent \textbf{Proof of $(2)$.} By the regularity conditions imposed in Assumptions (A$1$) and (A$2$), the monotonicity and continuity of the exponential function, and the independence of $\{\widetilde{R}_k\}_{k\in \{N+1,\ldots,n\}}$ where each has a finite fourth moment, we obtain
\begin{align}
	\label{eqn:second_moment}
	\mathbb{E}|\mathcal{T}_{n1}|^2&=\mathbb{E}\Bigg\{\big[J_n(x^n+\tau/\sqrt{n} \widetilde{R}^n)/J_n(x^n)-1\big]-\frac{\tau}{\sqrt{n}}\sum_{i\in V_n} \widetilde{R}_i\frac{D_{x_i} J_n(x^n)}{J_n(x^n)}\nonumber\\
	&\hspace{4cm}\quad-\frac{\tau^2}{2n} \sum_{i,j\in V_n} \widetilde{R}_i \widetilde{R}_j\frac{D_{x_ix_j} J_n(x^n)}{J_n(x^n)}\Bigg\}^2\longrightarrow 0.
\end{align}
Hence, one can show that $\mathcal{T}_{n1}\rightarrow 0$ in probability by Chebyshev's inequality.
\smallskip

\noindent \textbf{Proof of $(3)$.} To prove $\mathcal{T}_{n2}\rightarrow 0$ in probability, by Markov's inequality it suffices to show that $$\mathbb{E}|\mathcal{T}_{n2}|=\mathbb{E}|\mathcal{A}_{n2}^2-\mathcal{B}_{n2}^2|\rightarrow 0,$$
where
$$\mathcal{A}_{n2}:=\big[J_n(x^n+\tau/\sqrt{n} \widetilde{R}^n)/J_n(x^n)-1\big]$$
and
$$\mathcal{B}_{n2}:=\frac{\tau}{\sqrt{n}}\sum_{i\in V_n} \widetilde{R}_i\frac{D_{x_i} J_n(x^n)}{J_n(x^n)}.$$
Note that
\begin{align*}
	\mathbb{E}|\mathcal{A}_{n2}^2-\mathcal{B}_{n2}^2|=&\mathbb{E}|\mathcal{A}_{n2}-\mathcal{B}_{n2}||\mathcal{A}_{n2}+\mathcal{B}_{n2}|\\
	\leq &\Big[\mathbb{E}|\mathcal{A}_{n2}-\mathcal{B}_{n2}|^2\Big]^{1/2}\Big[\mathbb{E}|\mathcal{A}_{n2}+\mathcal{B}_{n2}|^2\Big]^{1/2}.
\end{align*}
We have $$\mathbb{E}|\mathcal{A}_{n2}-\mathcal{B}_{n2}|^2\rightarrow 0,$$
by the the same reasons used in establishing \eqref{eqn:second_moment}. Then as long as we can show that 
$$\mathbb{E}|\mathcal{A}_{n2}+\mathcal{B}_{n2}|^2<\infty,$$
we would complete the proof. To this end, given that
$$\mathbb{E}|\mathcal{A}_{n2}+\mathcal{B}_{n2}|^2\leq 2\mathbb{E}|\mathcal{A}_{n2}|^2+2\mathbb{E}|\mathcal{B}_{n2}|^2,$$  
both of which are finite by the the same reasons used in establishing \eqref{eqn:second_moment}.
\medskip

\noindent \textbf{Proof of $(4)$.} 
To prove $\mathcal{T}_{n3}\rightarrow 0$ in probability, since $s^2(\pi)$ is a constant, it suffices to show that 
\begin{align} 
	\label{eqn:proof4_suffice1}
	\frac{1}{n} \sum_{i,j\in V_n} \widetilde{R}_i \widetilde{R}_j\frac{D_{x_ix_j} J_n(x^n)}{J_n(x^n)} -\frac{1}{n}\left(\sum_{i\in V_n} \widetilde{R}_i\frac{D_{x_i} J_n(x^n)}{J_n(x^n)}\right)^2\rightarrow s^2(\pi)
\end{align}
in distribution.  Recalling that $J_n=e^{H_{V_n}}$ and
$$D_{x_i} J_n(x^n)=J_n(x^n)D_{x_i} H_{V_n}(x^n),$$
we have
$$D_{x_ix_j} J_n(x^n)=J_n(x^n)D_{x_i} H_{V_n}(x^n)D_{x_j} H_{V_n}(x^n)+J_n(x^n)D_{x_ix_j} H_{V_n}(x^n).$$
Given that
\begin{align*} 
	\frac{1}{n}\left(\sum_{i\in V_n} \widetilde{R}_i\frac{D_{x_i} J_n(x^n)}{J_n(x^n)}\right)^2=&\frac{1}{n}\sum_{i,j\in V_n} \widetilde{R}_i \widetilde{R}_j\frac{D_{x_i} J_n(x^n)}{J_n(x^n)} \frac{D_{x_j} J_n(x^n)}{J_n(x^n)}\\
	=&\frac{1}{n}\sum_{i,j\in V_n} \widetilde{R}_i \widetilde{R}_jD_{x_i} H_{V_n}(x^n)D_{x_j} H_{V_n}(x^n)
\end{align*}
and
\begin{align*} 
	&\frac{1}{n} \sum_{i,j\in V_n} \widetilde{R}_i \widetilde{R}_j\frac{D_{x_ix_j} J_n(x^n)}{J_n(x^n)}\\
	&=\frac{1}{n} \sum_{i,j\in V_n} \widetilde{R}_i \widetilde{R}_j D_{x_i} H_{V_n}(x^n)D_{x_j} H_{V_n}(x^n)+\frac{1}{n} \sum_{i,j\in V_n} \widetilde{R}_i \widetilde{R}_j D_{x_ix_j} H_{V_n}(x^n),
\end{align*}
we have
\begin{align*} 
	&\frac{1}{n} \sum_{i,j\in V_n} \widetilde{R}_i \widetilde{R}_j\frac{D_{x_ix_j} J_n(x^n)}{J_n(x^n)} -\frac{1}{n}\left(\sum_{i\in V_n} \widetilde{R}_i\frac{D_{x_i} J_n(x^n)}{J_n(x^n)}\right)^2\\
	&=\frac{1}{n} \sum_{i,j\in V_n} \widetilde{R}_i \widetilde{R}_j D_{x_ix_j} H_{V_n}(x^n).
\end{align*}
Hence, 
by \eqref{eqn:proof4_suffice1}, to finish the proof that $\mathcal{T}_{n3}\rightarrow 0$ in probability, it suffices to show that $\pi$ a.e.,
\begin{align*} 
	\frac{1}{n} \sum_{i,j\in V_n} \widetilde{R}_i \widetilde{R}_j D_{x_ix_j} H_{V_n}(x^n)\rightarrow s^2(\pi).
\end{align*}
This holds by Lemma $3$ (on page $195$) and Lemma $5$ (on page $196$) of \cite{Breyer2000From}, where the condition of bounded second-order derivatives is used but that is our Assumption (A$2$).
\medskip

\noindent \textbf{Proof of $(5)$.} 
Since $|\alpha_n|\leq 1$ and by Markov's inequality, we have 
\begin{align*}
	\mathbb{P}\left(\alpha_n\left|\frac{J_n(x^n+\tau/\sqrt{n} \widetilde{R}^n)}{J_n(x^n)}-1\right|^3\geq \epsilon\right)
	&\leq  \mathbb{P}\left(\left|\frac{J_n(x^n+\tau/\sqrt{n} \widetilde{R}^n)}{J_n(x^n)}-1\right|^3\geq \epsilon\right)\\
	&\leq  \epsilon^{-1}\mathbb{E}\left|\frac{J_n(x^n+\tau/\sqrt{n} \widetilde{R}^n)-J_n(x^n)}{J_n(x^n)}\right|^3,
\end{align*}
which goes to zero by the the same reasons used in establishing \eqref{eqn:second_moment}.

\subsubsection{Corollary}
\label{sec:Proofsa_corollary}

The following corollary is derived from the proof of Theorem \ref{thm:a_convergence} and will be utilized in the proof of Condition (M$1$) in the subsequent section.
\begin{corollary} \label{corollary}
	Under the assumptions imposed in Theorem \ref{thm:a_convergence}, with $N$ being a fixed positive integer, for almost every $r^N$ and $x$, as $n\rightarrow\infty$,
\begin{align*}
	&\frac{\frac{\psi_n(x^n-\tau n^{-1/2}\widetilde{R}^n)}{\psi_n(x^N+\tau n^{-1/2}r^N)}}{\frac{\psi_n(x^n)}{\psi_n(x^N)}} \Longrightarrow\exp\left(-\tau s(\pi)Z-(1/2)\tau^2 s^2(\pi)\right),\\ 
	&\frac{\frac{\psi_n(x^n-\tau n^{-1/2}\widetilde{R}^n)}{\psi_n(x^N-\tau n^{-1/2}r^N)}}{\frac{\psi_n(x^n)}{\psi_n(x^N)}} \Longrightarrow\exp\left(-\tau s(\pi)Z-(1/2)\tau^2 s^2(\pi)\right),
\end{align*}	
where $\widetilde{R}^n$ is defined in \eqref{eqn:widetildeR}, $Z$ is a standard normal random variable, and $s(\pi)$ is given in \eqref{eqn:s_pi_def}.
\end{corollary}

\begin{proof}
	Note that
\begin{align*}
	&\frac{\psi_n(x^n-\tau n^{-1/2}\widetilde{R}^n)\psi_n(x^N)}{\psi_n(x^N\pm \tau n^{-1/2}r^N)\psi_n(x^n)}\\
	&=\exp\left( -H_{V_n}(x^n-\tau/\sqrt{n} \widetilde{R}^n)+H_{V_n}(x^n)-H_{V_n}(x^N)+H_{V_n}(x^N\pm\tau/\sqrt{n} r^N)\right).
\end{align*}
In \eqref{eqn:cvg_Hamiltonian2}, we have that for almost every $r^N$ and $x$,
	\begin{align*}
	H_{V_n}(x^n+\tau/\sqrt{n} \widetilde{R}^n)-H_{V_n}(x^n)\Longrightarrow \tau s(\pi) Z+\tau^2s^2(\pi)/2.
\end{align*}
Since $\widetilde{R}^k$ has symmetric probability distribution for each $k\in \{N+1,\ldots, n\}$, slight modification of the proof of  \eqref{eqn:cvg_Hamiltonian2} would give
	\begin{align*}
	H_{V_n}(x^n-\tau/\sqrt{n} \widetilde{R}^n)-H_{V_n}(x^n)\Longrightarrow \tau s(\pi) Z+\tau^2s^2(\pi)/2.
\end{align*}
Next, 
	\begin{align*}
	H_{V_n}(x^N\pm \tau/\sqrt{n} r^N)-H_{V_n}(x^N)=\pm \frac{\tau}{\sqrt{n}}\sum_{i\in \{1,\ldots,N\}} \alpha r_i D_{x_i} H_{V_n}(x^N),
\end{align*}
where $|\alpha|\leq 1$. Since $N$ is fixed, by the assumed regularity conditions of $H_{V_n}$, we have
	\begin{align*}
	H_{V_n}(x^N\pm \tau/\sqrt{n} r^N)-H_{V_n}(x^N)\rightarrow 0.
\end{align*}
Considering that $e^{-x}$ is a monotone function, we complete the proof.


\end{proof}

\subsection{Proof of Condition (M$1$)}
\label{sec:Proof_M1}

In this section, we present the proof of Condition (M$1$). Before presenting the formal proof in Section \ref{sec:ProofsM1}, we outline the proof strategy in a more accessible manner in Section \ref{sec:ProofsM1_sketch} to enhance understanding.

\subsubsection{Sketch of the proof of Condition (M$1$)}
\label{sec:ProofsM1_sketch}

As the first work that introduced Mosco convergence to the optimal scaling problem, \cite{zanella2017dirichlet} creatively proposed a proof structure by separating the function convergence ($f_n$ to $f$) and the argument convergence ($X^n+\tau n^{-1/2}R^n$ to $X$) as $n$ goes to infinity. However, our proof strategy diverges from that of \cite{zanella2017dirichlet} due to two crucial challenges. Firstly, our definition of Mosco convergence differs from that of \cite{zanella2017dirichlet}, as we consider the sequence of Dirichlet forms on $\mathcal{H}_n= L^2(\mathbb{R}^{V_n};\pi_n)$ converging to a Dirichlet form on $\mathcal{H}= L^2(\mathbb{R}^V;\pi)$, whereas \cite{zanella2017dirichlet} always deals with sequences of Dirichlet forms in the infinite-dimensional space. Secondly, in this paper, we consider the target distribution in a general form with a large graph $V$, while \cite{zanella2017dirichlet} focused on the target distribution in the product form (our special case).


Condition (M$1$) in Definition \ref{def:Mosco_converges} states that if a sequence $\{f_n\}$ with $f_n\in \mathcal{H}_n$ converges to $f\in \mathcal{H}$ weakly in the sense of Definition \ref{def:convergences}, then $$\mathcal{E}(f)\leq\liminf_{n\rightarrow \infty}\mathcal{E}^n(f_n).$$ Our proof proceeds in five steps. In the first step, following \cite{zanella2017dirichlet}, we utilize the Cauchy-Schwarz inequality and the definition of the $L^2$ norm to place $\sqrt{\sE^n(f_n)}$ appropriately in the inequality \eqref{eqn:Cauchy-Schwarz_ratio}.  That is, as long as we can show that the fraction on which $\sqrt{\sE^n(f_n)}$ is larger than or equal to, converges to $\sqrt{\sE(f)}$, the mission to prove Condition (M$1$) would be accomplished. Handling the denominator is straightforward with the aid of Theorem \ref{thm:a_convergence}, while the numerator presents challenges addressed in Steps 2-4. Unlike most convergence analyses with fixed arguments for Dirichlet forms, the term $\sqrt{n}\Big[f_n(x^n+\tau n^{-1/2}r^n)-f_n(x^n)\Big]$ involves not only the convergence of $f_n$ to $f$ but also the convergence of $(x^n+\tau n^{-1/2}r^n)$ to $x$. Without special regularity conditions for $f_n$, directly obtaining the desired gradient $\nabla f(x)$ is unfeasible. Unfortunately, we lack these special regularity conditions. Therefore, in Step 2, through a change of variable argument, we transform $f_n(x^n+\tau n^{-1/2}r^n)$ into $f_n(x^n)$ and subsequently restructure the numerator into the sum of $\mathcal{A}_1$ and $\mathcal{A}_2$.

Our approach diverges from that of \cite{zanella2017dirichlet} from Step $2$.  In \cite{zanella2017dirichlet}, given that  $\psi_n$ is in the product form, they utilized
$$\frac{\psi_n(x^n-\tau n^{-1/2}r^n)}{\psi_n(x^N-\tau n^{-1/2}r^N)}=\psi_n(x^{n-N}-\tau n^{-1/2}r^{n-N})$$
and then transformed it into  $\psi_n(x^{n-N}+\tau n^{-1/2}r^{n-N})$ using reflection as the first two substeps in Step $2$. Unfortunately, with a general formed $\psi_n$, neither of these two substeps is applicable, and this divergence extends to the term  $\mathcal{A}_1$. Showing that  $\mathcal{A}_1\rightarrow 0$ as $n \rightarrow \infty$ is the objective of Step 3. To achieve this, we introduce an intermediate quantity to separate the crucial difference part in  $\mathcal{A}_1$
$$\Bigg\{\Bigg[\frac{\psi_n(x^N)}{\psi_n(x^N-\tau n^{-1/2}r^N)}\wedge \frac{\frac{\psi_n(x^n-\tau n^{-1/2}r^n)}{\psi_n(x^N-\tau n^{-1/2}r^N)}}{\frac{\psi_n(x^n)}{\psi_n(x^N)}}\Bigg] \\
-a_n(x^n, x^n+\tau n^{-1/2}r^n)\Bigg\},$$
by the triangle inequality. We then handle the easier term in  Step $3$. a), employing properties of the minimum function and the exponential function. The more challenging term is addressed in Step $3$. b), involving a meticulous analysis of two sets $\mathcal{S}_1$ and $\mathcal{S}_2$ defined after employing reflection. To control the corresponding terms, we flexibly use the following properties on the set $\mathcal{S}_2$ 
\begin{align*}
\Bigg\{ 1\leq \frac{\psi_n(x^n-\tau n^{-1/2}r^n)}{\psi_n(x^n)}\Bigg\}
	=&\Bigg\{ \frac{\psi_n(x^N)}{\psi_n(x^N-\tau n^{-1/2}r^N)}\leq \frac{\frac{\psi_n(x^n-\tau n^{-1/2}r^n)}{\psi_n(x^N-\tau n^{-1/2}r^N)}}{\frac{\psi_n(x^n)}{\psi_n(x^N)}}\Bigg\}\\
	=&\Bigg\{ \frac{\psi_n(x^N)}{\psi_n(x^N+\tau n^{-1/2}r^N)}\leq \frac{\frac{\psi_n(x^n-\tau n^{-1/2}r^n)}{\psi_n(x^N+\tau n^{-1/2}r^N)}}{\frac{\psi_n(x^n)}{\psi_n(x^N)}}\Bigg\},
\end{align*}
and then use Corollary \ref{corollary} to handle the above quantities.

After establishing $\mathcal{A}_1\rightarrow 0$, in Step 4, we demonstrate that
\begin{align*}
	\mathcal{A}_2 \longrightarrow -\dfrac{\tau c(\tau)}{\sqrt{2}}\mathbb{E}\left(f(X) \left[\nabla_{x^N} \Big(\xi(X^N,R^N)\psi(X^N)\Big)\right]^T R^N \psi^{-1}(X^N)\right).
\end{align*}
This is obtained by the imposed regularity of $\psi_n$, Theorem \ref{thm:a_convergence}, the smoothness of $\xi$, the compact support on $K$, and importantly Lemma \ref{useful_cvg_equi} regarding the weak convergence of $f_n\in \mathcal{H}_n$ to $f\in \mathcal{H}$ in the sense of Definition \ref{def:convergences}. 
In Step $5$, we combine all the results and show that the fraction obtained in Step $1$ converges to $\sqrt{\sE(f)}$, that is, $\dfrac{\tau \sqrt{c(\tau)}}{\sqrt{2}}\sqrt{\mathbb{E}|\nabla f(X)|^2}$. The gradient form $\nabla f(x)$ is obtained by an integration-by-parts argument using the compact support with the gradient term $\nabla_{x^N} \Big(\xi(X^N,R^N)\psi(X^N)\Big)$ in the above equation, for $f\in S$ where $S$ is given in \eqref{def:S}. For the other case where $f\in \mathcal{H}\backslash S$ such that $\sE(f)=\infty$, \cite{zanella2017dirichlet} beautifully solved it regardless of the specific forms of the target distribution, so we refer to the proof therein.

\subsubsection{Formal proof of Condition (M$1$)}
\label{sec:ProofsM1}


We start by defining a proper test function such that the dimension of its arguments does not grow with $n$.
Fix a positive integer $N$ and choose a non-zero test function $\xi$ in $C_0^{\infty}(\mathbb{R}^{2N})$, which is the class of infinitely differentiable functions with compact support on $\mathbb{R}^{2N}$. We further fix a compact set $K\subset \mathbb{R}^{2N}$ such that $\cup_{n\in \mathbb{N}}\{(x^N,r^N): \xi(x^N-\tau n^{-1/2}r^N,r^N)>0\}\subseteq K$ (see page $4069$ of \cite{zanella2017dirichlet} for an example of $K$). Note that the function $\xi(X^N,R^N)\mathbbm{1}_{\{U<a_n(X^n, X^n+\tau n^{-1/2}R^n)\}}$ belongs to $L^2_{(X,R,U)}$ and is non-zero. We are going to prove Condition (M$1$) through the following $5$ steps.
\medskip

\noindent\textbf{Step $1$.} In this step, using the test function $\xi$, we establish the proof structure by placing $\mathcal{E}^n: \mathcal{H}_n\rightarrow \overline{\mathbb{R}}$ in the desired inequality form in \eqref{eqn:Cauchy-Schwarz_ratio} below.
Recall that the Dirichlet form $\sE^n(f_n)$ defined in \eqref{En} is given by
\begin{align*}
	\sE^n(f_n)=\dfrac{n}{2}\mathbb{E}\Big[f_n(X^n(1))-f_n(X^n(0))\Big]^2.
\end{align*}
Now, set 
$$\Psi_n(f_n)=\sqrt{\dfrac{n}{2}}\Big[f_n(X^n(1))-f_n(X^n(0))\Big].$$ 
Applying the Cauchy-Schwarz inequality, we obtain
\begin{align}\label{eqn:Cauchy-Schwarz_ratio}
	\sqrt{\sE^n(f_n)}=&\left\| \Psi_n(f_n)\right\|_{L^2_{(X,R,U)}}\nonumber\\
	\geq & \frac{\left\langle\Psi_n(f_n), \xi(X^N,R^N)\mathbbm{1}_{\{U<a_n(X^n, X^n+\tau n^{-1/2}R^n)\}}\right\rangle_{L^2_{(X,R,U)}}}{\left\|  \xi(X^N,R^N)\mathbbm{1}_{\{U<a_n(X^n, X^n+\tau n^{-1/2}R^n)\}} \right\|_{L^2_{(X,R,U)}}}.
\end{align}
Here, $U$ is the Uniform $(0,1)$ random variable that is independent of $X$ and $R$. Then intergating out $U$ yields,
\begin{align*}
&\left\|  \xi(X^N,R^N)\mathbbm{1}_{\{U<a_n(X^n, X^n+\tau n^{-1/2}R^n)\}} \right\|_{L^2_{(X,R,U)}}\\
&= \sqrt{\mathbb{E}\Big[ \xi^2(X^N,R^N) a_n(X^n, X^n+\tau n^{-1/2}R^n)\Big]}\\
&= \sqrt{\mathbb{E}\Big[ \xi^2(X^N,R^N)\mathbb{E}\big[a_n(X^n, X^n+\tau n^{-1/2}R^n)\mid X^N,R^N\big] \Big] }.
\end{align*}
Thus, by Theorem \ref{thm:a_convergence}, the denominator in \eqref{eqn:Cauchy-Schwarz_ratio} can be handled as follows:
\begin{align}
	\label{eqn:denominator}
	&\left\|  \xi(X^N,R^N)\mathbbm{1}_{\{U<a_n(X^n, X^n+\tau n^{-1/2}R^n)\}} \right\|_{L^2_{(X,R,U)}}\nonumber\\
	&\hspace{3cm}\longrightarrow \sqrt{c(\tau)}\|  \xi(X^N,R^N) \|_{L^2_{(X,R)}}.
\end{align}
The numerator in \eqref{eqn:Cauchy-Schwarz_ratio}, after intergating out $U$, becomes
\begin{align*}
	&\sqrt{\dfrac{n}{2}}\int_{\mathbb{R}^{V_n}\times \mathbb{R}^{V_n}} \Big[f_n(x^n+\tau n^{-1/2}r^n)-f_n(x^n)\Big] a_n(x^n, x^n+\tau n^{-1/2}r^n) \\
	&\hspace{5.8cm}\times \xi(x^N,r^N)\varphi^{n}(r^n)dr^n\psi_n(x^n)dx^n,
\end{align*}
which will be handled in Steps $2-4$. 
\medskip

\noindent\textbf{Step $2$.} In this step, we transform $f_n(x^n+\tau n^{-1/2}r^n)$ into $f_n(x^n)$ using a change of variable argument to reformulate the numerator in \eqref{eqn:Cauchy-Schwarz_ratio}, as analyzing the asymptotic behavior of $\sqrt{n}\Big[f_n(x^n+\tau n^{-1/2}r^n)-f_n(x^n)\Big]$ with both changing functions and changing arguments is challenging.
We express the numerator in \eqref{eqn:Cauchy-Schwarz_ratio} as
\begin{align*}
&\sqrt{\dfrac{n}{2}}\int_{\mathbb{R}^{V_n}\times \mathbb{R}^{V_n}} \Big[f_n(x^n+\tau n^{-1/2}r^n)-f_n(x^n)\Big] a_n(x^n, x^n+\tau n^{-1/2}r^n) \\
	&\hspace{5cm}\times \xi(x^N,r^N)\varphi^{n}(r^n)dr^n\psi_n(x^n)dx^n\\
&=\sqrt{\dfrac{n}{2}}\int_{\mathbb{R}^{V_n}\times \mathbb{R}^{V_n}}f_n(x^n) a_n(x^n-\tau n^{-1/2}r^n, x^n)\xi(x^N-\tau n^{-1/2}r^N,r^N) \\
&\hspace{5cm}\times \varphi^{n}(r^n)dr^n\psi_n(x^n-\tau n^{-1/2}r^n)dx^n\\	
&\quad -\sqrt{\dfrac{n}{2}}\int_{\mathbb{R}^{V_n}\times \mathbb{R}^{V_n}} f_n(x^n) a_n(x^n, x^n+\tau n^{-1/2}r^n)\xi(x^N,r^N)\varphi^{n}(r^n)dr^n\\
&\hspace{5cm}\times \psi_n(x^n)dx^n.
\end{align*}
For the first term of the above equation, we have
\begin{align*}
&\sqrt{\dfrac{n}{2}}\int_{\mathbb{R}^{V_n}\times \mathbb{R}^{V_n}}f_n(x^n) a_n(x^n-\tau n^{-1/2}r^n, x^n) \xi(x^N-\tau n^{-1/2}r^N,r^N)\\
	&\hspace{4cm}\times \varphi^{n}(r^n)dr^n\psi_n(x^n-\tau n^{-1/2}r^n)dx^n\\	
&=\sqrt{\dfrac{n}{2}}\int_{\mathbb{R}^{V_n}\times \mathbb{R}^{V_n}}f_n(x^n) \Bigg[1\wedge \frac{\psi_n(x^n)}{\psi_n(x^n-\tau n^{-1/2}r^n)}\Bigg] \xi(x^N-\tau n^{-1/2}r^N,r^N)\\
&\hspace{4cm}\times \varphi^{n}(r^n)dr^n\psi_n(x^n-\tau n^{-1/2}r^n)dx^n\\	
&=\sqrt{\dfrac{n}{2}}\int_{\mathbb{R}^{V_n}\times \mathbb{R}^{V_n}}f_n(x^n) \Bigg[\psi_n(x^n)\wedge \psi_n(x^n-\tau n^{-1/2}r^n)\Bigg]\\
&\hspace{4cm}\times \xi(x^N-\tau n^{-1/2}r^N,r^N) \varphi^{n}(r^n)dr^ndx^n\\	
&=\sqrt{\dfrac{n}{2}}\int_{\mathbb{R}^{V_n}\times \mathbb{R}^{V_n}}f_n(x^n) \Bigg[1\wedge \frac{\psi_n(x^n-\tau n^{-1/2}r^n)}{\psi_n(x^n)}\Bigg] \xi(x^N-\tau n^{-1/2}r^N,r^N)\\
&\hspace{4cm}\times \varphi^{n}(r^n)dr^n\psi_n(x^n)dx^n\\	
&=\sqrt{\dfrac{n}{2}}\int_{\mathbb{R}^{V_n}\times \mathbb{R}^{V_n}}f_n(x^n) \Bigg[\frac{\psi_n(x^N)}{\psi_n(x^N-\tau n^{-1/2}r^N)}\wedge \frac{\frac{\psi_n(x^n-\tau n^{-1/2}r^n)}{\psi_n(x^N-\tau n^{-1/2}r^N)}}{\frac{\psi_n(x^n)}{\psi_n(x^N)}}\Bigg] \\
&\hspace{0.85cm}\times \frac{\psi_n(x^N-\tau n^{-1/2}r^N)}{\psi_n(x^N)}\xi(x^N-\tau n^{-1/2}r^N,r^N)\varphi^{n}(r^n)dr^n\psi_n(x^n)dx^n,
\end{align*}
where in the final step, we isolated the term $\frac{\psi_n(x^N-\tau n^{-1/2}r^N)}{\psi_n(x^N)}$, which will assist in bounding the $\sqrt{n}$ factor later.

Up to this point, we can represent the numerator of \eqref{eqn:Cauchy-Schwarz_ratio} as follows:
\begin{align}\label{eqn:terms_in_numerator}
	&\sqrt{\dfrac{n}{2}}\int_{\mathbb{R}^{V_n}\times \mathbb{R}^{V_n}} \Big[f_n(x^n+\tau n^{-1/2}r^n)-f_n(x^n)\Big] a_n(x^n, x^n+\tau n^{-1/2}r^n) \xi(x^N,r^N)\nonumber\\
	&\hspace{7.8cm}\times \varphi^{n}(r^n)dr^n\psi_n(x^n)dx^n\nonumber\\
	&=\sqrt{\dfrac{n}{2}}\int_{\mathbb{R}^{V_n}\times \mathbb{R}^{V_n}}f_n(x^n) \Bigg\{\Bigg[\frac{\psi_n(x^N)}{\psi_n(x^N-\tau n^{-1/2}r^N)}\wedge \frac{\frac{\psi_n(x^n-\tau n^{-1/2}r^n)}{\psi_n(x^N-\tau n^{-1/2}r^N)}}{\frac{\psi_n(x^n)}{\psi_n(x^N)}}\Bigg]\nonumber\\
	&\hspace{4.6cm}\times \frac{\psi_n(x^N-\tau n^{-1/2}r^N)}{\psi_n(x^N)}\xi(x^N-\tau n^{-1/2}r^N,r^N)\nonumber\\
	&\hspace{4 cm}-a_n(x^n, x^n+\tau n^{-1/2}r^n)\xi(x^N,r^N)\Bigg\}\nonumber\\
	&\hspace{7cm}\times \varphi^{n}(r^n)dr^n\psi_n(x^n)dx^n\nonumber\\	
	&=:\mathcal{A}_1+\mathcal{A}_2,
\end{align}
where 
\begin{align*}
	\mathcal{A}_1=&\dfrac{\tau}{\sqrt{2}}\frac{\sqrt{n}}{\tau}\int_{\mathbb{R}^{V_n}\times \mathbb{R}^{V_n}}f_n(x^n) \Bigg\{\Bigg[\frac{\psi_n(x^N)}{\psi_n(x^N-\tau n^{-1/2}r^N)}\wedge \frac{\frac{\psi_n(x^n-\tau n^{-1/2}r^n)}{\psi_n(x^N-\tau n^{-1/2}r^N)}}{\frac{\psi_n(x^n)}{\psi_n(x^N)}}\Bigg] \\
	&\hspace{3.3cm}-a_n(x^n, x^n+\tau n^{-1/2}r^n)\Bigg\}\frac{\psi_n(x^N-\tau n^{-1/2}r^N)}{\psi_n(x^N)}\\
	&\hspace{2.2cm}\times \xi(x^N-\tau n^{-1/2}r^N,r^N)\varphi^{n}(r^n)dr^n\psi_n(x^n)dx^n
\end{align*}
and
\begin{align*}
	\mathcal{A}_2=&\sqrt{\dfrac{n}{2}}\int_{\mathbb{R}^{V_n}\times \mathbb{R}^{V_n}}f_n(x^n)a_n(x^n, x^n+\tau n^{-1/2}r^n)\\
	&\hspace{1cm}\times \Bigg\{\xi(x^N-\tau n^{-1/2}r^N,r^N)\frac{\psi_n(x^N-\tau n^{-1/2}r^N)}{\psi_n(x^N)} -\xi(x^N,r^N)\Bigg\}\\
	&\hspace{6cm}\times \varphi^{n}(r^n)dr^n \psi_n(x^n)dx^n.
\end{align*}
\smallskip

\noindent\textbf{Step $3$.} In this step, we aim to show that $\mathcal{A}_1\rightarrow 0$ as $n \rightarrow \infty$.
By Lemma \ref{useful_cvg_equi}, we have that $\sup_n\|f_n\|_{\mathcal{H}_n}<\infty$. Additionally, given the continuity of the functions $\xi$ and $\psi_n$, we can assert on the compact set $K$ that
$$\xi(x^N-\tau n^{-1/2}r^N,r^N)\frac{\psi_n(x^N-\tau n^{-1/2}r^N)}{\psi_n(x^N)}<\infty.$$ Hence, the following term in $\mathcal{A}_1$
$$\dfrac{\tau}{\sqrt{2}}f_n(x^n)\xi(x^N-\tau n^{-1/2}r^N,r^N)\frac{\psi_n(x^N-\tau n^{-1/2}r^N)}{\psi_n(x^N)}$$
is bounded in the $L^2_{(X,R)}$ norm.

Next, we aim to show that the remaining term of $\mathcal{A}_1$ 
\begin{align*}
\mathcal{A}_{1,2}&:=\frac{\sqrt{n}}{\tau}\int_{\mathbb{R}^{V_n}\times \mathbb{R}^{V_n}}\Bigg\{\Bigg[\frac{\psi_n(x^N)}{\psi_n(x^N-\tau n^{-1/2}r^N)}\wedge \frac{\frac{\psi_n(x^n-\tau n^{-1/2}r^n)}{\psi_n(x^N-\tau n^{-1/2}r^N)}}{\frac{\psi_n(x^n)}{\psi_n(x^N)}}\Bigg] \\
&\hspace{2.3cm}-a_n(x^n, x^n+\tau n^{-1/2}r^n)\Bigg\} \varphi^{n}(r^n)dr^n\psi_n(x^n)dx^n \\
&\longrightarrow  0.
\end{align*}
Note that
\begin{align*}
	\mathcal{A}_{1,2}=&\frac{\sqrt{n}}{\tau}\int_{\mathbb{R}^{V_n}\times \mathbb{R}^{V_n}}\Bigg\{\Bigg[\frac{\psi_n(x^N)}{\psi_n(x^N-\tau n^{-1/2}r^N)}\wedge \frac{\frac{\psi_n(x^n-\tau n^{-1/2}r^n)}{\psi_n(x^N-\tau n^{-1/2}r^N)}}{\frac{\psi_n(x^n)}{\psi_n(x^N)}}\Bigg] \\
	&\hspace{2.4cm}-1\wedge \frac{\psi_n(x^n+\tau n^{-1/2}r^n)}{\psi_n(x^n)}\Bigg\} \varphi^{n}(r^n)dr^n\psi_n(x^n)dx^n \\
	=&\frac{\sqrt{n}}{\tau}\int_{\mathbb{R}^{V_n}\times \mathbb{R}^{V_n}}\Bigg\{\Bigg[\frac{\psi_n(x^N)}{\psi_n(x^N-\tau n^{-1/2}r^N)}\wedge \frac{\frac{\psi_n(x^n-\tau n^{-1/2}r^n)}{\psi_n(x^N-\tau n^{-1/2}r^N)}}{\frac{\psi_n(x^n)}{\psi_n(x^N)}}\Bigg]\\
	&\hspace{2.4cm}-\Bigg[\frac{\psi_n(x^N+\tau n^{-1/2}r^N)}{\psi_n(x^N)}\wedge \frac{\frac{\psi_n(x^n-\tau n^{-1/2}r^n)}{\psi_n(x^N-\tau n^{-1/2}r^N)}}{\frac{\psi_n(x^n)}{\psi_n(x^N)}}\Bigg]\Bigg\}\\
	&\hspace{5cm}\times\varphi^{n}(r^n)dr^n\psi_n(x^n)dx^n \\
	&+\frac{\sqrt{n}}{\tau}\int_{\mathbb{R}^{V_n}\times \mathbb{R}^{V_n}}\Bigg\{\Bigg[\frac{\psi_n(x^N+\tau n^{-1/2}r^N)}{\psi_n(x^N)}\wedge \frac{\frac{\psi_n(x^n-\tau n^{-1/2}r^n)}{\psi_n(x^N-\tau n^{-1/2}r^N)}}{\frac{\psi_n(x^n)}{\psi_n(x^N)}}\Bigg]\\
	&\hspace{2.5cm}-1\wedge \frac{\psi_n(x^n+\tau n^{-1/2}r^n)}{\psi_n(x^n)}\Bigg\}\varphi^{n}(r^n)dr^n\psi_n(x^n)dx^n \\
	=: & \mathcal{A}_{1,2,1}+\mathcal{A}_{1,2,2},
\end{align*}
which will be handled in the following Steps $3$. a) and $3$. b) respectively.
\medskip

\noindent\textbf{Step $3$. a).} Here, we show that $\mathcal{A}_{1,2,1}\rightarrow 0$ as $n\rightarrow \infty$. By the fact that for any $a,b,c>0$ one has
$$|(a\wedge c)-(b\wedge c) |\leq |a-b|,$$
together with the fact that $e^x$ is locally Lipschitz,
the term in $\mathcal{A}_{1,2,1}$
\begin{align*}
	&\frac{\sqrt{n}}{\tau}\Bigg|\Bigg[\frac{\psi_n(x^N)}{\psi_n(x^N-\tau n^{-1/2}r^N)}\wedge \frac{\frac{\psi_n(x^n-\tau n^{-1/2}r^n)}{\psi_n(x^N-\tau n^{-1/2}r^N)}}{\frac{\psi_n(x^n)}{\psi_n(x^N)}}\Bigg]\\
	&\hspace{2.4cm}-\Bigg[\frac{\psi_n(x^N+\tau n^{-1/2}r^N)}{\psi_n(x^N)}\wedge \frac{\frac{\psi_n(x^n-\tau n^{-1/2}r^n)}{\psi_n(x^N-\tau n^{-1/2}r^N)}}{\frac{\psi_n(x^n)}{\psi_n(x^N)}}\Bigg]\Bigg|\\
	&\leq \frac{\sqrt{n}}{\tau}\Bigg|\frac{\psi_n(x^N)}{\psi_n(x^N-\tau n^{-1/2}r^N)}-\frac{\psi_n(x^N+\tau n^{-1/2}r^N)}{\psi_n(x^N)}\Bigg|\\
	&= \frac{\sqrt{n}}{\tau} \Bigg|\exp\Big(-H_n(x^N)+H_n(x^N-\tau n^{-1/2}r^N)\Big)\\
	&\hspace{2.4cm}-\exp\Big(-H_n(x^N+\tau n^{-1/2}r^N)+H_n(x^N)\Big)\Bigg|\\
	&\leq \frac{\sqrt{n}}{\tau} \Bigg|H_n(x^N+\tau n^{-1/2}r^N)+H_n(x^N-\tau n^{-1/2}r^N)-2H_n(x^N)\Bigg|\\
	&=\frac{\sqrt{n}}{\tau}\frac{\tau^2}{n} \left|\alpha\sum_{i,j\in \{1,\ldots,N\}} R_iR_j D_{x_ix_j} H_{V_n}(x^N)\right|,
\end{align*}
where $|\alpha|\leq 1$.
Since $N$ is a fixed constant, and by the uniformly bounded second-order derivative assumption, we have $\mathcal{A}_{1,2,1}\rightarrow 0$ as $n\rightarrow 0$ on the compact set $K$.
\medskip

\noindent\textbf{Step $3$. b).} Now, we show that $\mathcal{A}_{1,2,2}\rightarrow 0$ as $n\rightarrow \infty$.
Due to the symmetry of $\varphi^{n}$, we can establish that
\begin{align*}
\mathcal{A}_{1,2,2}
=&\frac{\sqrt{n}}{\tau}\int_{\mathbb{R}^{V_n}\times \mathbb{R}^{V_n}}\Bigg[\frac{\psi_n(x^N+\tau n^{-1/2}r^N)}{\psi_n(x^N)}\wedge \frac{\frac{\psi_n(x^n-\tau n^{-1/2}r^n)}{\psi_n(x^N-\tau n^{-1/2}r^N)}}{\frac{\psi_n(x^n)}{\psi_n(x^N)}}\Bigg]\\
&\hspace{6cm}\times\varphi^{n}(r^n)dr^n\psi_n(x^n)dx^n \\
	&-\frac{\sqrt{n}}{\tau}\int_{\mathbb{R}^{V_n}\times \mathbb{R}^{V_n}}\Bigg[1\wedge \frac{\psi_n(x^n+\tau n^{-1/2}r^n)}{\psi_n(x^n)}\Bigg]\varphi^{n}(r^n)dr^n\psi_n(x^n)dx^n \\
	=&\frac{\sqrt{n}}{\tau}\int_{\mathbb{R}^{V_n}\times \mathbb{R}^{V_n}}\Bigg[\frac{\psi_n(x^N+\tau n^{-1/2}r^N)}{\psi_n(x^N)}\wedge \frac{\frac{\psi_n(x^n-\tau n^{-1/2}r^n)}{\psi_n(x^N-\tau n^{-1/2}r^N)}}{\frac{\psi_n(x^n)}{\psi_n(x^N)}}\Bigg]\\
	&\hspace{6cm}\times\varphi^{n}(r^n)dr^n\psi_n(x^n)dx^n \\
	&-\frac{\sqrt{n}}{\tau}\int_{\mathbb{R}^{V_n}\times \mathbb{R}^{V_n}}\Bigg[1\wedge \frac{\psi_n(x^n-\tau n^{-1/2}r^n)}{\psi_n(x^n)}\Bigg]\varphi^{n}(r^n)dr^n\psi_n(x^n)dx^n.
\end{align*}
Set $$\mathcal{S}_1:=\Bigg\{ \frac{\psi_n(x^N+\tau n^{-1/2}r^N)}{\psi_n(x^N)}\leq \frac{\frac{\psi_n(x^n-\tau n^{-1/2}r^n)}{\psi_n(x^N-\tau n^{-1/2}r^N)}}{\frac{\psi_n(x^n)}{\psi_n(x^N)}}\Bigg\}$$
and 
\begin{align*}
	\mathcal{S}_2:=&\Bigg\{ 1\leq \frac{\psi_n(x^n-\tau n^{-1/2}r^n)}{\psi_n(x^n)}\Bigg\}\\
	=&\Bigg\{ \frac{\psi_n(x^N)}{\psi_n(x^N-\tau n^{-1/2}r^N)}\leq \frac{\frac{\psi_n(x^n-\tau n^{-1/2}r^n)}{\psi_n(x^N-\tau n^{-1/2}r^N)}}{\frac{\psi_n(x^n)}{\psi_n(x^N)}}\Bigg\}\\
	=&\Bigg\{ \frac{\psi_n(x^N)}{\psi_n(x^N+\tau n^{-1/2}r^N)}\leq \frac{\frac{\psi_n(x^n-\tau n^{-1/2}r^n)}{\psi_n(x^N+\tau n^{-1/2}r^N)}}{\frac{\psi_n(x^n)}{\psi_n(x^N)}}\Bigg\}.
\end{align*}
Then for the following term in $\mathcal{A}_{1,2,2}$,
\begin{align*}
\widetilde{\mathcal{A}}_{1,2,2}:=&\frac{\sqrt{n}}{\tau}\int_{ \mathcal{S}_1 }\Bigg[\frac{\psi_n(x^N+\tau n^{-1/2}r^N)}{\psi_n(x^N)}\Bigg]\varphi^{n-N}(r^{n-N})dr^{n-N}\\
	&+\frac{\sqrt{n}}{\tau}\int_{ \mathcal{S}_1^c}\frac{\frac{\psi_n(x^n-\tau n^{-1/2}r^n)}{\psi_n(x^N-\tau n^{-1/2}r^N)}}{\frac{\psi_n(x^n)}{\psi_n(x^N)}}\varphi^{n-N}(r^{n-N})dr^{n-N} \\
	&-\frac{\sqrt{n}}{\tau}\int_{ \mathcal{S}_2}\varphi^{n-N}(r^{n-N})dr^{n-N} \\
	&-\frac{\sqrt{n}}{\tau}\int_{ \mathcal{S}_2^c}\Bigg[\frac{\psi_n(x^n-\tau n^{-1/2}r^n)}{\psi_n(x^n)}\Bigg]\varphi^{n-N}(r^{n-N})dr^{n-N}\\
	=&\frac{\sqrt{n}}{\tau}\Bigg[\frac{\psi_n(x^N+\tau n^{-1/2}r^N)}{\psi_n(x^N)}-1\Bigg]\Bigg\{\int_{ \mathcal{S}_1 }\varphi^{n-N}(r^{n-N})dr^{n-N}\\
	&\hspace{1.5cm}-\int_{ \mathcal{S}_2^c}\frac{\frac{\psi_n(x^n-\tau n^{-1/2}r^n)}{\psi_n(x^N+\tau n^{-1/2}r^N)}}{\frac{\psi_n(x^n)}{\psi_n(x^N)}}\varphi^{n-N}(r^{n-N})dr^{n-N}\Bigg\}\\
	&+\frac{\sqrt{n}}{\tau}\Bigg\{\int_{ \mathcal{S}_1 }\varphi^{n-N}(r^{n-N})dr^{n-N}\\
	&\hspace{1.5cm}-\int_{ \mathcal{S}_2^c}\frac{\frac{\psi_n(x^n-\tau n^{-1/2}r^n)}{\psi_n(x^N+\tau n^{-1/2}r^N)}}{\frac{\psi_n(x^n)}{\psi_n(x^N)}}\varphi^{n-N}(r^{n-N})dr^{n-N}\Bigg\}\\
	&+\frac{\sqrt{n}}{\tau}\int_{ \mathcal{S}_1^c}\frac{\frac{\psi_n(x^n-\tau n^{-1/2}r^n)}{\psi_n(x^N-\tau n^{-1/2}r^N)}}{\frac{\psi_n(x^n)}{\psi_n(x^N)}}\varphi^{n-N}(r^{n-N})dr^{n-N} \\
	&-\frac{\sqrt{n}}{\tau}\int_{ \mathcal{S}_2} \varphi^{n-N}(r^{n-N})dr^{n-N}.
\end{align*}
Supposing that $$\psi_n(x^N-\tau n^{-1/2}r^N)\geq \psi_n(x^N+\tau n^{-1/2}r^N),$$ 
we can handle the opposite case similarly. Thus, we obtain
$$0<\frac{\frac{\psi_n(x^n-\tau n^{-1/2}r^n)}{\psi_n(x^N-\tau n^{-1/2}r^N)}}{\frac{\psi_n(x^n)}{\psi_n(x^N)}}\leq \frac{\frac{\psi_n(x^n-\tau n^{-1/2}r^n)}{\psi_n(x^N+\tau n^{-1/2}r^N)}}{\frac{\psi_n(x^n)}{\psi_n(x^N)}},$$
which yields
\begin{align*}
\widetilde{\mathcal{A}}_{1,2,2}\leq &\frac{\sqrt{n}}{\tau}\Bigg[\frac{\psi_n(x^N+\tau n^{-1/2}r^N)}{\psi_n(x^N)}-1\Bigg]\Bigg\{\int_{ \mathcal{S}_1 }\varphi^{n-N}(r^{n-N})dr^{n-N}\\
	&\hspace{1.7cm}-\int_{ \mathcal{S}_2^c}\frac{\frac{\psi_n(x^n-\tau n^{-1/2}r^n)}{\psi_n(x^N+\tau n^{-1/2}r^N)}}{\frac{\psi_n(x^n)}{\psi_n(x^N)}}\varphi^{n-N}(r^{n-N})dr^{n-N}\Bigg\}\\
	&+\frac{\sqrt{n}}{\tau}\int_{ \mathcal{S}_1^c \cap  \mathcal{S}_2}\Bigg[\frac{\frac{\psi_n(x^n-\tau n^{-1/2}r^n)}{\psi_n(x^N-\tau n^{-1/2}r^N)}}{\frac{\psi_n(x^n)}{\psi_n(x^N)}}-1\Bigg]\varphi^{n-N}(r^{n-N})dr^{n-N}.
\end{align*}

Since on the set $\mathcal{S}_1^c \cap \mathcal{S}_2$, we have
\begin{align*}
\Bigg|\frac{\frac{\psi_n(x^n-\tau n^{-1/2}r^n)}{\psi_n(x^N-\tau n^{-1/2}r^N)}}{\frac{\psi_n(x^n)}{\psi_n(x^N)}}-1\Bigg|\leq \mathcal{C}_{\operatorname{max}}^N,
\end{align*}
where 
\begin{align*}
\mathcal{C}_{\operatorname{max}}^N:=\max\Bigg(\Bigg|\frac{\psi_n(x^N+\tau n^{-1/2}r^N)}{\psi_n(x^N)}-1\Bigg|,\; \Bigg|\frac{\psi_n(x^N)}{\psi_n(x^N-\tau n^{-1/2}r^N)}-1\Bigg|\Bigg).
\end{align*}
we can bound $\widetilde{\mathcal{A}}_{1,2,2}$ as follows:
\begin{align*}
\widetilde{\mathcal{A}}_{1,2,2}\leq & \frac{\sqrt{n}}{\tau}\mathcal{C}_{\operatorname{max}}^N\Bigg(\Bigg|\int_{ \mathcal{S}_1 }\varphi^{n-N}(r^{n-N})dr^{n-N}\\
	&\hspace{1.7cm}-\int_{ \mathcal{S}_2^c}\frac{\frac{\psi_n(x^n-\tau n^{-1/2}r^n)}{\psi_n(x^N+\tau n^{-1/2}r^N)}}{\frac{\psi_n(x^n)}{\psi_n(x^N)}}\varphi^{n-N}(r^{n-N})dr^{n-N}\Bigg|\\
	&\hspace{1.7cm}+\int_{ \mathcal{S}_1^c \cap  \mathcal{S}_2}\varphi^{n-N}(r^{n-N})dr^{n-N} \Bigg).
\end{align*}	
Note that as $n \rightarrow \infty$, we have 
\begin{align}
	\label{eqn:bound_limit}
\frac{\psi_n(x^N+\tau n^{-1/2}r^N)}{\psi_n(x^N)}\rightarrow 1 \quad \text{and}\quad \frac{\psi_n(x^N)}{\psi_n(x^N-\tau n^{-1/2}r^N)}\rightarrow 1.
\end{align}	
 Therefore, 
\begin{align*}
	\int_{ \mathcal{S}_1^c \cap  \mathcal{S}_2}\varphi^{n-N}(r^{n-N})dr^{n-N} \rightarrow 0.
\end{align*}	

Next, we aim to show that
\begin{align*}
&\int_{ \mathcal{S}_1 }\varphi^{n-N}(r^{n-N})dr^{n-N}-\int_{ \mathcal{S}_2^c}\frac{\frac{\psi_n(x^n-\tau n^{-1/2}r^n)}{\psi_n(x^N+\tau n^{-1/2}r^N)}}{\frac{\psi_n(x^n)}{\psi_n(x^N)}}\varphi^{n-N}(r^{n-N})dr^{n-N}\longrightarrow 0.
\end{align*}	
It holds 
by \eqref{eqn:bound_limit}, Corollary \ref{corollary},
 and the following equation (on page $4073$ of \cite{zanella2017dirichlet})
\begin{align*}
&\int_{y>0} \exp\left(-\frac{(y+(1/2)\tau^2 s^2(\pi))^2}{2\tau^2 s^2(\pi)}\right)dy\\
&\hspace{3cm}-\int_{y\leq 0} e^y\exp \left(-\frac{(y+(1/2)\tau^2 s^2(\pi))^2}{2\tau^2 s^2(\pi)}\right)dy=0.
\end{align*}

Up to this point,  it suffices to show that $\frac{\sqrt{n}}{\tau}\mathcal{C}_{\operatorname{max}}^N$ is bounded on the compact set $K$, to complete the proof of $\widetilde{\mathcal{A}}_{1,2,2}\rightarrow 0$. To this end, note that
 \begin{align*}
 	&\frac{\sqrt{n}}{\tau}\max\Bigg(\Bigg|\frac{\psi_n(x^N+\tau n^{-1/2}r^N)}{\psi_n(x^N)}-1\Bigg|,\; \Bigg|\frac{\psi_n(x^N)}{\psi_n(x^N-\tau n^{-1/2}r^N)}-1\Bigg|\Bigg)\\
 	&\leq C \sqrt{n}\frac{1}{\sqrt{n}} \sum_{i=1}^N \left|r_i D_{x_i} \psi_n(x^N)\right|.
 \end{align*}
Here, $C$ represents an absolute constant independent of  $n$, $N$ is a fixed constant, and $r_i D_{x_i} \psi_n(x^N)$ for each $i$ is bounded on the compact set $K$ owing to the continuity of the first-order derivatives of $\psi_n$. Finally, by the law of total expectation,  $\widetilde{\mathcal{A}}_{1,2,2}\rightarrow 0$ gives $\mathcal{A}_{1,2,2}\rightarrow 0$.
\medskip

\noindent\textbf{Step $4$.} 
In this step, we investigate the asymptotic behavior of $\mathcal{A}_2$, defined in  \eqref{eqn:terms_in_numerator}, 
 as $n$ goes to infinity. Notably,  $\mathcal{A}_2$ can be rewritten as
\begin{align*}
	\mathcal{A}_2=&-\dfrac{\tau}{\sqrt{2}}\int_{\mathbb{R}^{V_n}\times \mathbb{R}^{V_n}}f_n(x^n)a_n(x^n, x^n+\tau n^{-1/2}r^n)\\
	&\hspace{0.3cm}\times \Bigg\{\frac{\xi(x^N-\tau n^{-1/2}r^N,r^N)\psi_n(x^N-\tau n^{-1/2}r^N)-\xi(x^N,r^N)\psi_n(x^N)}{-\tau n^{-1/2}\psi_n(x^N)} \Bigg\}\\
	&\hspace{7.5cm}\times \varphi^{n}(r^n)dr^n \psi_n(x^n)dx^n.
\end{align*}
Note that the imposed regularity of $\psi_n$ implies that it is strictly positive and bounded with bounded first-order derivative on the compact set $K$. Consequently,  by Theorem \ref{thm:a_convergence}, the smoothness of $\xi$, and the compact support on $K$, we have
\begin{align}\label{eqn:pointwise_cvg}
& \Bigg\{\frac{\xi(x^N-\tau n^{-1/2}r^N,r^N)\psi_n(x^N-\tau n^{-1/2}r^N)-\xi(x^N,r^N)\psi_n(x^N)}{-\tau n^{-1/2}\psi_n(x^N)} \Bigg\}\nonumber\\
	&\hspace{2cm}\times \int_{\mathbb{R}^{n-N}}a_n(x^n, x^n+\tau n^{-1/2}r^n)\varphi^{n-N}(r^{n-N})dr^{n-N}\nonumber\\
& \longrightarrow c(\tau)\frac{\left[\nabla_{x^N} \Big(\xi(x^N,r^N)\psi(x^N)\Big)\right]^Tr^N}{\psi(x^N)},
\end{align}
pointwisely, where $\psi$ is the density of $\pi$ and coincides with $\psi_n$ on a fixed number of arguments $N\leq n$. Therefore, this expression is bounded by
\begin{align*}
	& \sup_{(x^N,r^N)\in K} \left|\nabla_{x^N} \Big(\xi(x^N,r^N)\psi_n(x^N)\Big)\right|\times \sup_{(x^N,r^N)\in K}\left\{ \frac{|r^N|}{ \psi_n(x^N) }\right\}\\
	&\hspace{0.5cm}\times \limsup_{n\rightarrow \infty}\int_{\mathbb{R}^{n-N}}a_n(x^n, x^n+\tau n^{-1/2}r^n)\varphi^{n-N}(r^{n-N})dr^{n-N}.
\end{align*}
Hence, the convergence in \eqref{eqn:pointwise_cvg} is also in $L^2(X,W)$. Recalling that $f_n\in \mathcal{H}_n$ converges to $f\in \mathcal{H}$ weakly, we have by
Lemma \ref{useful_cvg_equi} that 
\begin{align}
	\label{eqn:numeratorA2}
	\mathcal{A}_2 \longrightarrow -\dfrac{\tau c(\tau)}{\sqrt{2}}\mathbb{E}\left(f(X) \left[\nabla_{x^N} \Big(\xi(X^N,R^N)\psi(X^N)\Big)\right]^T R^N \psi^{-1}(X^N)\right).
\end{align}
\smallskip

\noindent\textbf{Step $5$.} 
Plugging $\mathcal{A}_1\rightarrow 0$, as achieved in Step $3$,  and \eqref{eqn:numeratorA2} into \eqref{eqn:terms_in_numerator}, we derive the following result for the numerator in \eqref{eqn:Cauchy-Schwarz_ratio}:
\begin{align}\label{eqn:numerator}
	&\sqrt{\dfrac{n}{2}}\int_{\mathbb{R}^{V_n}\times \mathbb{R}^{V_n}} \Big[f_n(x^n+\tau n^{-1/2}r^n)-f_n(x^n)\Big] a_n(x^n, x^n+\tau n^{-1/2}r^n) \xi(x^N,r^N)\nonumber\\
	&\hspace{7cm}\times \varphi^{n}(r^n)dr^n\psi_n(x^n)dx^n\nonumber\\
	&\longrightarrow -\dfrac{\tau c(\tau)}{\sqrt{2}}\mathbb{E}\left(f(X) \left[\nabla_{x^N} \Big(\xi(X^N,R^N)\psi(X^N)\Big)\right]^T R^N \psi^{-1}(X^N)\right).
\end{align}
Substituting  \eqref{eqn:numerator} and \eqref{eqn:denominator} into \eqref{eqn:Cauchy-Schwarz_ratio}, we obtain
\begin{align*}
	&\liminf_{n\rightarrow\infty}\sqrt{\sE^n(f_n)}\\
	&\hspace{1cm}\geq  \frac{-\dfrac{\tau c(\tau)}{\sqrt{2}}\mathbb{E}\left(f(X) \left[\nabla_{x^N} \Big(\xi(X^N,R^N)\psi(X^N)\Big)\right]^T R^N \psi^{-1}(X^N)\right)}{\sqrt{c(\tau)}\|  \xi(X^N,R^N) \|_{L^2_{(X,R)}}}.
\end{align*}
If $f \in S$, where $S$ is defined in \eqref{def:S}, an integration-by-parts argument, utilizing the compact support, yields
\begin{align*}
	&-\mathbb{E}\left(f(x) \left[\nabla_{x^N} \Big(\xi(X^N,R^N)\psi_n(X^N)\Big)\right]^T R^N \psi^{-1}(X^N)\right)\\
	&\hspace{3cm}=\mathbb{E}\left( \xi(X^N,R^N)\left[\Big(\nabla_{x^N} f(X)\Big)^T R^N\right] \right)\\
	&\hspace{3cm}=\mathbb{E}\left( \xi(X^N,R^N)\left[\Big(\nabla f(X)\Big)^T R\right] \right),
\end{align*}
where the last equality is by the independence of $\{R_i\}_i$. That is, we in fact have
\begin{align*}
	\liminf_{n\rightarrow\infty}	\sqrt{\sE^n(f_n)}\geq & \dfrac{\tau \sqrt{c(\tau)}}{\sqrt{2}}\sup_{N\geq 1}\sup_{\substack{\xi\in C_0^{\infty}(\mathbb{R}^{2N})\\ \xi\neq 0}}\frac{\mathbb{E}\left( \xi(X^N,R^N)\left[\Big(\nabla f(X)\Big)^T R\right] \right)}{\|  \xi(X^N,R^N) \|_{L^2_{(X,R)}}}\\
	=& \dfrac{\tau \sqrt{c(\tau)}}{\sqrt{2}}\left\|\Big(\nabla f(X)\Big)^T R\right\|_{L^2_{(X,R)}}\\
	=& \dfrac{\tau \sqrt{c(\tau)}}{\sqrt{2}}\sqrt{\mathbb{E}|\nabla f(X)|^2}\\
	=& \sqrt{\sE(f)},
\end{align*}
where the second equality is by the independence of $X$ and $R$, along with the second moment of each component of $R$ being one. 

Recall that 
$$\mathcal{E}(\cdot): \mathcal{H}\rightarrow \overline{\mathbb{R}},\quad\quad\sE(f) = \left\{ \begin{array}{lcl}
	\sE(f), && \mbox{for}
	\; f\in S, \\ 
	\infty, && \mbox{for}\; f \in \mathcal{H}\backslash S.
\end{array}\right.$$
Thus far, we have concluded the proof of (M$1$) for the major case where $f \in S$. The proof for the other case, $f \in \mathcal{H}\backslash S$ where $\sE(f)=\infty$, is omitted, as it can be followed using the outlined proof on pages $4073-4074$ of \cite{zanella2017dirichlet}.

\subsection{Proof of Condition (M$2$)}
\label{sec:Proof_M2}

In this section, we provide the proof of Condition (M$2$). Prior to delving into the formal proof in Section \ref{sec:ProofsM2}, we provide an accessible outline of the proof strategy in Section \ref{sec:ProofsM2_sketch} to aid understanding.

\subsubsection{Sketch of the proof of Condition (M$2$)}
\label{sec:ProofsM2_sketch}

Condition (M$2$) in Definition \ref{def:Mosco_converges} states that for every $f\in \mathcal{H}$, there exists a sequence $\{f_n\}$ with $f_n\in \mathcal{H}_n$ converging to $f$ strongly in the sense of Definition \ref{def:convergences}, such that 
$$\mathcal{E}(f)=\lim_{n\rightarrow \infty}\mathcal{E}^n(f_n).$$
Thus, our goal is to construct such a weakly converging sequence $\{f_n\}$.  
In contrast to \cite{zanella2017dirichlet}, the challenge here lies in handling functions in converging Hilbert spaces. To address this, we employ the sequence $\Gamma_n$ defined in Definition \ref{def:Hiblertspaces_cvg}, which is asymptotically close to a unitary operator (as noted on page $611$ of \cite{kuwae2003convergence}), such that $\Gamma_n f=f$ for all $f\in \mathcal{H}$. 
Then, we use the for every $f\in S$, there exists a converging sequence $\{\widehat{f}_m\}$ such that $\mathcal{E}(\widehat{f}_m)\rightarrow \mathcal{E}(f)$ where $\widehat{f}_m\in \mathscr{F}C_{b}^\infty$ defined in \eqref{eqn:mathscr_F} that is dense in $S$. We complete the proof using the triangle inequality after establishing $\sE^n(\Gamma_n \widehat{f}_m)\rightarrow\sE(\widehat{f}_m)$.

\subsubsection{Formal proof of Condition (M$2$)}
\label{sec:ProofsM2} 

For any $f\in \mathcal{H}$, if $f\in \mathcal{H}\backslash S$, take $f_n:=\Gamma_n f$. Clearly $f_n\in \mathcal{H}_n$ converges to $f\in \mathcal{H}$ weakly. Proof of Condition (M$2$) can be finished by following from Condition (M$1$) that
$$\infty=\sE(f)\leq \liminf_{n\rightarrow\infty} \sE^n(f_n)=\lim_{n\rightarrow\infty} \sE^n(f_n)=\infty.$$
In the following, we only consider $f\in S$. 

Recall that $\mathscr{F}C_{b}^\infty$ defined in \eqref{eqn:mathscr_F} is dense in $S$ with respect to the $\sE_1^{1/2}$ norm. For every $f\in \mathcal{H}$, one can choose a sequence $\{\widehat{f}_m\}$ where $\widehat{f}_m\in \mathscr{F}C_{b}^\infty$ such that as $m\rightarrow \infty,$
$$\widehat{f}_m\rightarrow f \quad \text{and}\quad \mathcal{E}(\widehat{f}_m)\rightarrow \mathcal{E}(f).$$
For any fixed $m$, noting that $\widehat{f}_m\in \mathscr{F}C_{b}$ depends on a finite number of components. Without loss of generality, we suppose it depends on the first $N\geq 1$ components.
By the definition of $\mathcal{E}^n$ given in \eqref{En} and the law of total expectation,
\begin{align*}
\sE^n(\Gamma_n \widehat{f}_m)
&=\frac{n}{2}\mathbb{E}\Big[\big(\Gamma_n \widehat{f}_m(X^N+\tau n^{-1/2}R^N)-\Gamma_n \widehat{f}_m(X^N)\big)^2\\
&\hspace{2cm}\times\mathbb{E}\big[ a_n(X^n, X^n+\tau n^{-1/2}R^n)|X^N, R^N \big]\Big]\\
&=\frac{\tau^2}{2}\mathbb{E}\Bigg[\Bigg(\frac{\Gamma_n \widehat{f}_m(X^N+\tau n^{-1/2}R^N)-\Gamma_n \widehat{f}_m(X^N)}{\tau n^{-1/2}}\Bigg)^2 \\
&\hspace{2cm}\times\mathbb{E}\big[ a_n(X^n, X^n+\tau n^{-1/2}R^n)|X^N, R^N \big]\Bigg].
\end{align*}
Note that in the above equation, since the function $a_n$ is bounded by $1$, the expression inside the outer expectation is bounded by $(|R^N|\|\nabla(\Gamma_n \widehat{f}_m)\|_{\infty})^2$ which is an integrable random variable. By  Theorem \ref{thm:a_convergence} and the regularity of $\widehat{f}_m$, together with the fact that $\Gamma_n$ is asymptotically close to a unitary operator, we have
\begin{align*}
&\Bigg(\frac{\Gamma_n \widehat{f}_m(X^N+\tau n^{-1/2}R^N)-\Gamma_n \widehat{f}_m(X^N)}{\tau n^{-1/2}}\Bigg)^2\\
&\hspace{4cm}\times\mathbb{E}\big[ a_n(X^n, X^n+\tau n^{-1/2}R^n)|X^n, R^N \big]\\
&\longrightarrow ((\nabla \widehat{f}_m (X^N))^TR^N)^2c(\tau)
\end{align*}
pointwise as $n\rightarrow \infty$. Therefore, by the dominated convergence theorem, we have
\begin{align*}
\sE^n(\Gamma_n \widehat{f}_m)\longrightarrow \frac{\tau^2c(\tau)}{2}\mathbb{E}\Big((\nabla \widehat{f}_m (X^N))^TR^N\Big)^2&=\frac{\tau^2c(\tau)}{2}\mathbb{E}\Big|\nabla \widehat{f}_m (X^N)\Big|^2\\
&=\sE(\widehat{f}_m),
\end{align*}
where the equality is by the independence between $X^N$ and $R^N=\{R_i\}_{i=1,\ldots,N}$, and the second moment of each $R_i$ being one.

By Proposition $7.2$ (on page $255$) of \cite{Kolesnikov2005Convergence}, the space $\cup_{n=1}^{\infty}\mathcal{H}_n$ is metrizable by some metric $d$. 
Let $\{\mathcal{M}(m)\}$ be a sequence of natural numbers such that $\mathcal{M}(m+1)>\mathcal{M}(m)$, 
$$d(\widehat{f}_m, \Gamma_n \widehat{f}_m)\leq 1/m \quad\text{and}\quad
|\sE^n(\Gamma_n \widehat{f}_m)-\sE(\widehat{f}_m)|\leq 1/m,$$
for any $n>\mathcal{M}(m)$.
To complete the proof of Condition (M$2$), it suffices to construct a sequence $\{f_n\}$ with $f_n\in \mathcal{H}_n$ converging to $f$ strongly such that 
$$\mathcal{E}(f)=\lim_{n\rightarrow \infty}\mathcal{E}^n(f_n).$$
Now, set $f_n=\Gamma_n \widehat{f}_{k(n)}$, where $k(n)$ is chosen in such a way that
$\mathcal{M}(k(n))<n\leq \mathcal{M}(k(n)+1)$ if $n>\mathcal{M}(2)$ and $k(n)=1$ otherwise. Hence, we have
$f_n\in \mathcal{H}_n$, and $f_n\rightarrow f$ strongly, and by the triangle inequality
\begin{align*}
|\sE^n(f_n)-\sE(f)|&=|\sE^n(\Gamma_n \widehat{f}_{k(n)})-\sE(f)|\\
&\leq |\sE^n(\Gamma_n \widehat{f}_{k(n)})-\sE(\widehat{f}_{k(n)})|+|\sE(\widehat{f}_{k(n)})-\sE(f)|,
\end{align*}
which yields
 $\sE^n(f_n)\rightarrow \sE(f)$ as $n\rightarrow \infty$.

\section{Discussion and conclusion}
\label{sec:Discussion_and_conclusion}

In this paper, we introduced Mosco convergence of Dirichlet forms on changing Hilbert spaces to analyze the optimal scaling of the RWM algorithm for the Gibbs measure on large graphs. We demonstrated the significant advantages of the Dirichlet form approach over the standard diffusion approach. In Corollary \ref{thm:corollary}, we established that the limiting Dirichlet form $\mathcal{E}$ coincides with the limiting diffusion $X$. Consequently, the standard optimal scaling procedure described in Section \ref{sec:Optimal_scaling} can be applied. We acknowledge that this paper does not provide a full generalization of \cite{Breyer2000From}. They also discussed the uniqueness of the Gibbs measure $\pi$, which involves phase transitions, while our focus here is on advocating the Dirichlet form approach and comparing it with the diffusion approach. Therefore, phase transition analysis is beyond the scope of this paper, and we assumed the existence of only one Gibbs measure $\pi$. Furthermore, we acknowledge that \cite{zanella2017dirichlet} paved the way for using Dirichlet forms on optimal scaling problems. 

We believe that the strategies developed in this paper could assist in reducing the regularity assumptions imposed in the aforementioned literature with dependent distributions. In addition to exploring a different model, suggested by one referee, another avenue for future research could involve developing a general update proposal using correlated $R^n$ instead of them being i.i.d. We are not aware of any optimal scaling paper in that setting, which could potentially present another limitation when employing the diffusion approach. It is worth mentioning that conducting a comprehensive analysis would require significant computational resources and examinations for high-dimensional MCMC algorithms. Given that the Gibbs distribution encompasses the normal distribution, for which optimal scaling analysis results are well-known, and considering our focus on the comparison of theoretical analysis methods, we defer the numerical study to future research. 

It is worth mentioning that the distributions considered in \cite{durmus2017optimal} can be either nondifferentiable at some points or supported on an interval, through introducing the concept of locally asymptotic normality (LAN) into this framework \citep{le2012asymptotic, bickel1993efficient}. 
Our proof of Theorem \ref{thm:a_convergence} shares some similarity with LAN and rescaled LAN \citep{ning2021scalable}, which generalizes LAN to an enlarged neighborhood and contributes to Monte Carlo log-likelihood evaluations when the number of observations is large. 
Rescaled LAN might be applicable to product-formed distributions using the approach developed in \cite{durmus2017optimal}, thus contributing to further optimal scaling analysis. For general distributions, second-order derivatives are still required in our analysis. However, \cite{Breyer2000From} and \cite{yang2020optimal} demand strong third-order conditions.  Our model corresponds to the factor graph models of \cite{yang2020optimal} when the number of cliques $k$ goes to infinity in equation (41) therein; see equation \eqref{eqn:generalU} for illustration. This aspect was not addressed in their work, as they solely focused on finite cases. 
The uniformly bounded cliques in \cite{yang2020optimal} are comparable to the spatial homogeneity enforced in this paper from a mathematical analysis perspective, but that might be advantageous to broaden the underlying graph structures.

\appendix

\section{Numerical demonstration}
\label{sec:Numerical}
The family of generalized normal distributions has garnered significant attention from the engineering community due to the flexibility of its parametric form in modeling various physical phenomena \citep{dytso2018analytical}. In this section, we focus on the generalized normal distribution, which has the following PDF:
\[
f(x \mid \mu, \alpha, \beta) = \frac{\beta}{2\alpha \Gamma\left(1/\beta\right)} \exp\left( - \left(|x - \mu|/\alpha\right)^\beta \right),
\]	
where we set \(\mu = 0\), \(\alpha = 1\), and \(\beta =3\). It does not satisfy the growth condition specified by Hypothesis (H$5$) in \cite{Breyer2000From}; however, we have relaxed this requirement by utilizing the convergence of Dirichlet forms in Section \ref{sec:Mosco_convergence}. The Python code that reproduces our numerical analyses is publicly available at https://github.com/patning/MCMC\_Optimal\_Scaling.

\begin{figure*}[t!]
	\centering
	\includegraphics[width = 5in]{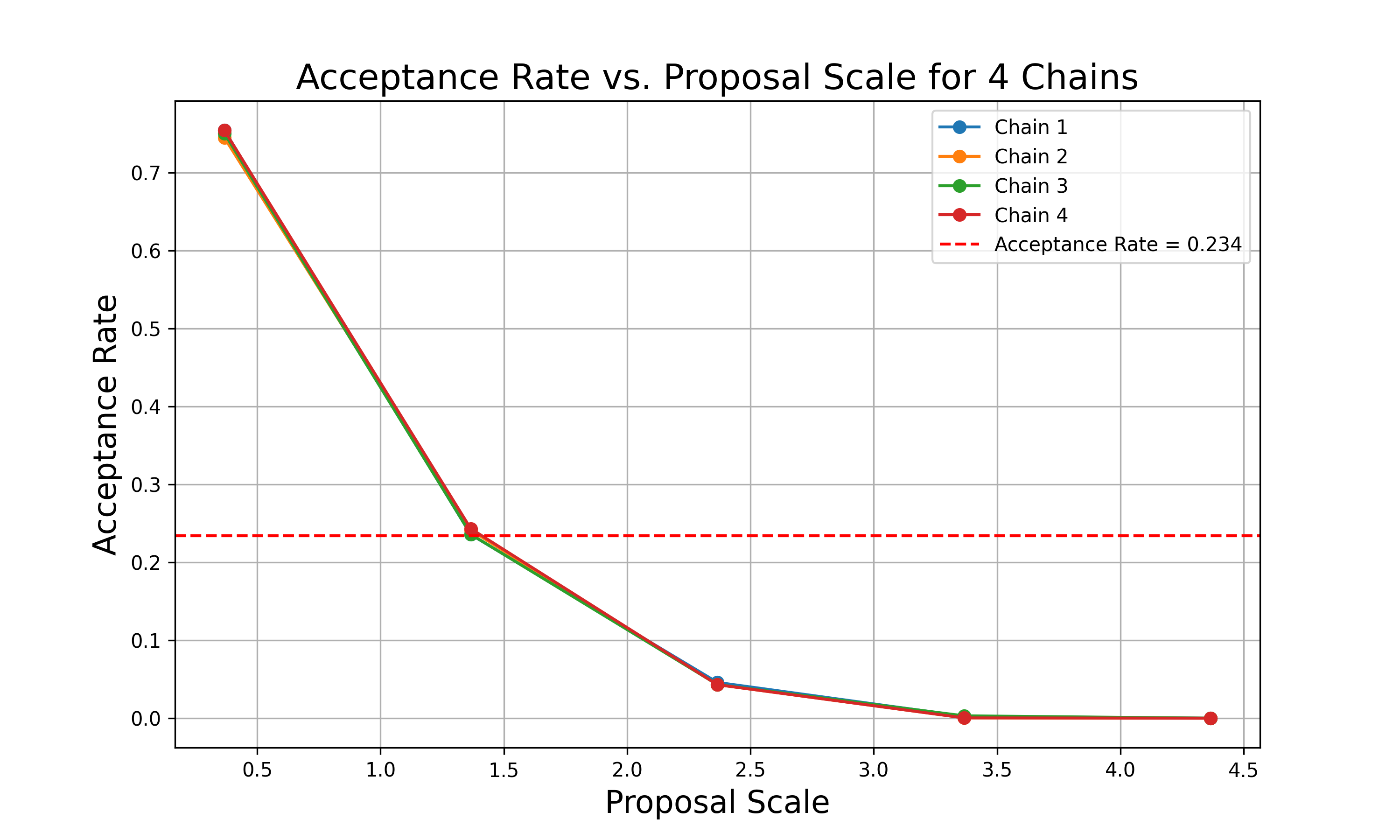}
	\caption{Acceptance rate versus proposal scale for four Markov chains. Each curve represents a different chain, and the horizontal line at 0.234 marks the optimal acceptance rate for the RWM algorithm.
	}
	\label{fig:accept_prob}
\end{figure*}

We employed the RWM algorithm to sample from  a 100-dimensional target distribution, where the above generalized normal distribution with \(\mu = 0\), \(\alpha = 1\), and \(\beta =3\) apply to each dimension independently. Four independent MCMC chains were generated, each starting from a 100-dimensional zero vector. For each chain, we run 20,000 iterations of the RWM algorithm, discarding the first 5,000 samples as a burn-in period to ensure that the chains were not influenced by the starting point.  The proposal distribution used was a multivariate normal distribution centered at the current state with variance-covariance matrix being the identity matrix. Five proposal scales $[0.3666,\, 1.366,\, 2.366,\, 3.366,\, 4.366]$ were tested as visualized in Figures \ref{fig:accept_prob} and \ref{fig:ESJD}, where $1.366$ is calculated by the formula $\tau^*\approx 2.38 /s(\pi)$ in Section \ref{sec:Optimal_scaling}. Specifically, in this example, $s(\pi)$ defined in equation \eqref{eqn:s_pi_def} is given by
\begin{align*}
s(\pi)=\left\{\mathbb{E}\left[\frac{d}{dx} \log\big(f(x \mid \mu, \alpha, \beta)\big)\right]^2\right\}^{1/2}=3[\mathbb{E}X^4]^{1/2}&=3[\Gamma(5/3)/\gamma(1/3)]^{1/2}\\
&\approx 1.7415,
\end{align*}
by the moment result in \cite{nadarajah2005generalized}.

\begin{figure*}[t!]
	\centering
	\includegraphics[width = 5in]{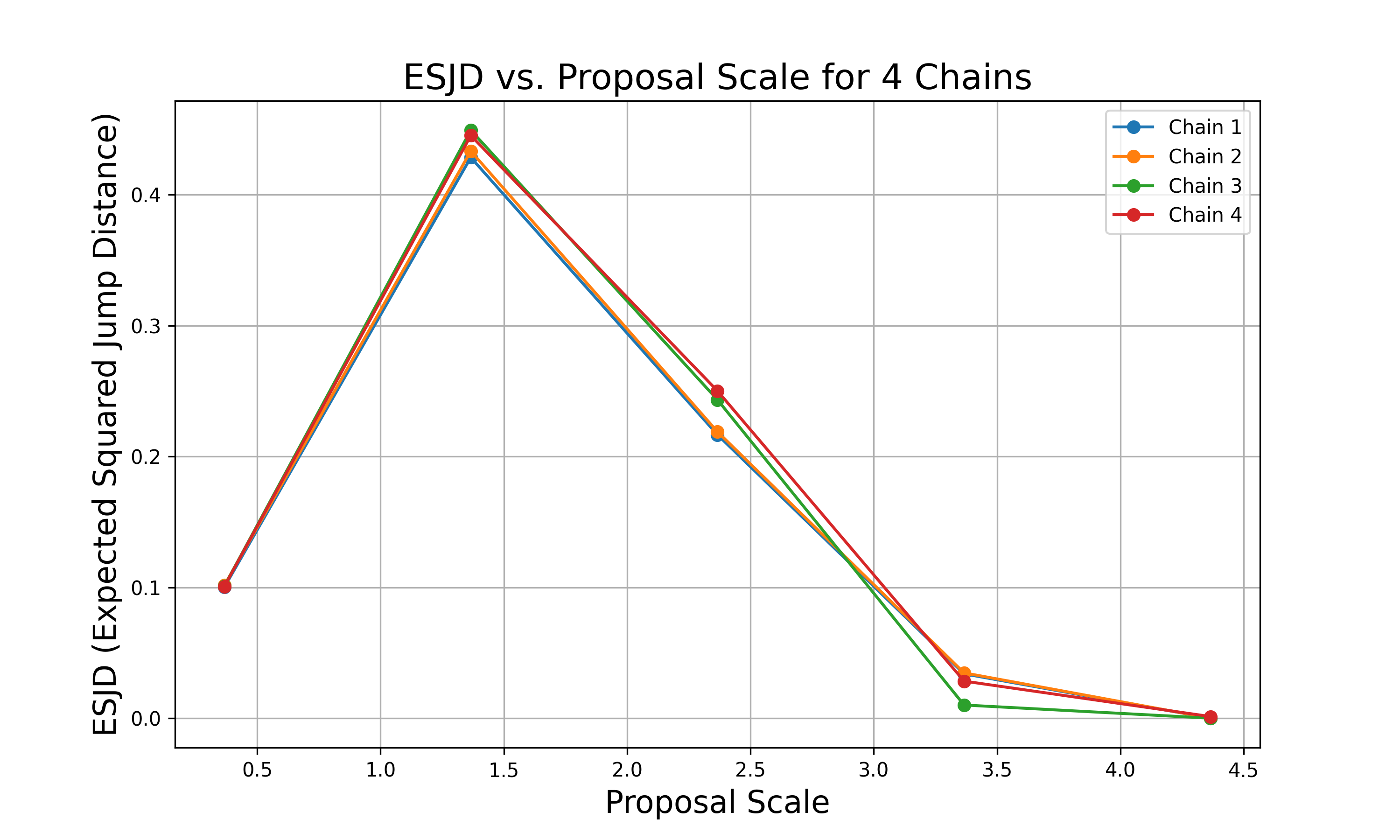}
	\caption{ESJD  versus  proposal scale for four Markov chains. Each curve represents the ESJD behavior for a different chain with varying proposal scales.}
	\label{fig:ESJD}
\end{figure*}

Figure \ref{fig:accept_prob} reports the acceptance rate as a function of the proposal scale for four Markov chains. As expected, the acceptance rate decreases as the proposal scale increases. When the proposal scale is too small, the proposed value is very close to the current state, making it easier to accept. However, this leads to a strong correlation between the current and newly accepted states, causing the Markov chain to mix very slowly. On the other hand, when the proposal scale is too large, the proposed value is far from the current state, resulting in low acceptance probabilities. In this case, the chain can also mix slowly due to repeated rejections, which maintain identical states across iterations. 
Hence, the acceptance rate alone should not be the sole criterion for evaluating the performance of a MCMC algorithm.

\begin{figure*}[t!]
	\centering
	\includegraphics[width = 5in]{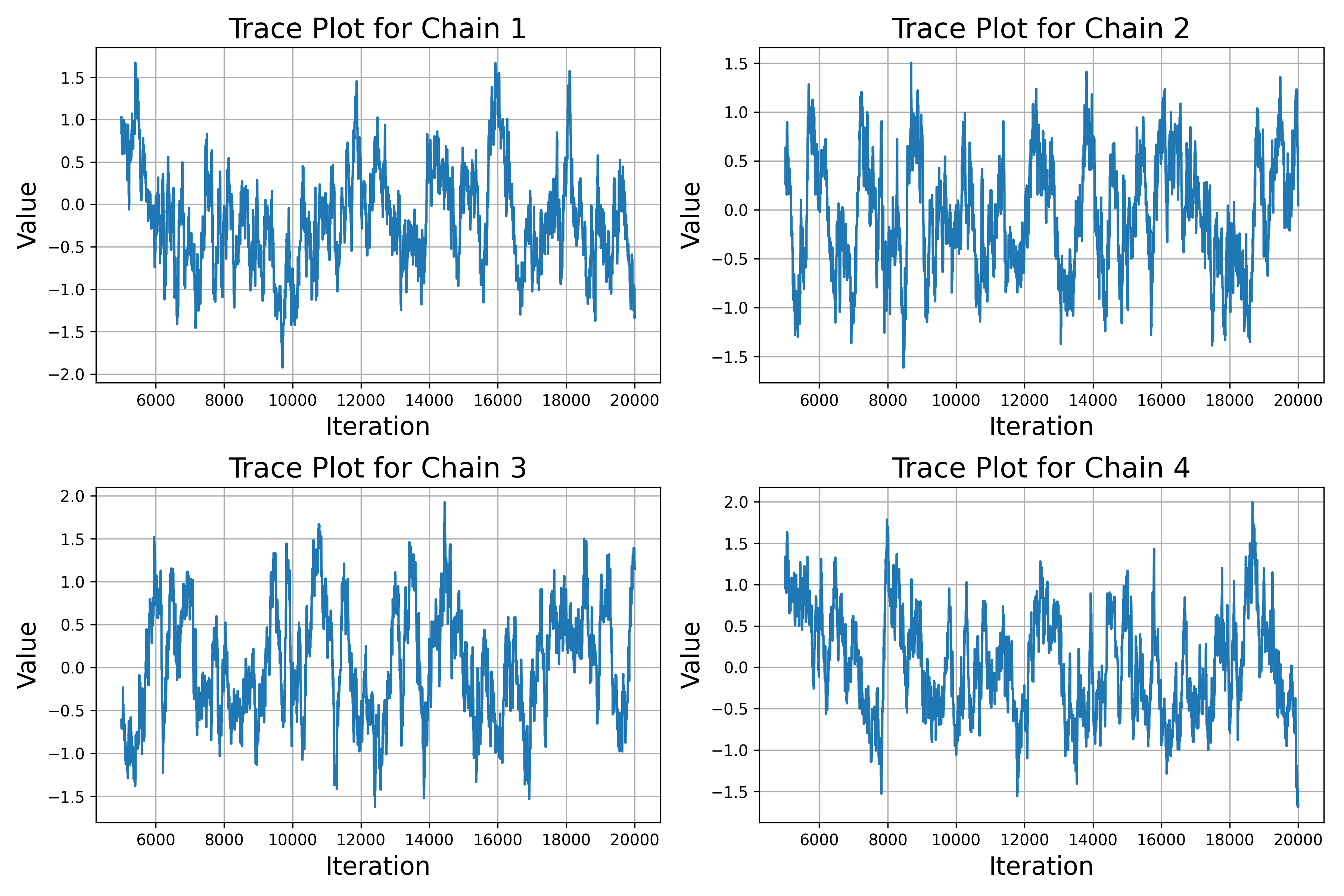}
	\caption{Trace plots for four Markov chains, showing 15,000 iterations after discarding the initial 5,000 as burn-in.}
	\label{fig:trace}
\end{figure*}
Figure \ref{fig:ESJD} reports the Expected Squared Jump Distance (ESJD) as a function of the proposal scale for four Markov chains.  ESJD is a commonly used metric for evaluating the efficiency of a MCMC algorithm. It is defined as the expected value of the squared distance between consecutive states of the chain:
\[
\text{ESJD} = \mathbb{E} \| X^n(t+1) - X^n(t) \|^2,
\]
where \(X^n(t)\) and \(X^n(t+1)\) represent the current and proposed states of the chain, respectively. The ESJD measures how far the Markov chain moves between consecutive steps, providing a direct way to assess the exploration of the target distribution. From Figure \ref{fig:ESJD}, we can see that the optimal scale value $\tau^*=1.366$ gives the highest ESJD  across all four Markov chains. Additionally, the corresponding acceptance rates are approximately 0.234 in Figure \ref{fig:accept_prob}, aligning with the theoretical optimal rate for efficient exploration in high-dimensional MCMC algorithms. 

Figure \ref{fig:trace} reports the trace plots for the first dimension of the four Markov chains using the optimal scale value $\tau^*=1.366$. These plots were generated over a total of 20,000 iterations with an initial burn-in period of 5,000 iterations, revealing the convergence behavior of each chain toward the target distribution. The trace plots illustrate the sampled values over time, highlighting the chains' mixing and exploration capabilities. Ideally, the traces exhibit a randomized pattern around the target distribution, indicating effective exploration. 

\begin{figure*}[t!]
	\centering
	\includegraphics[width = 5in]{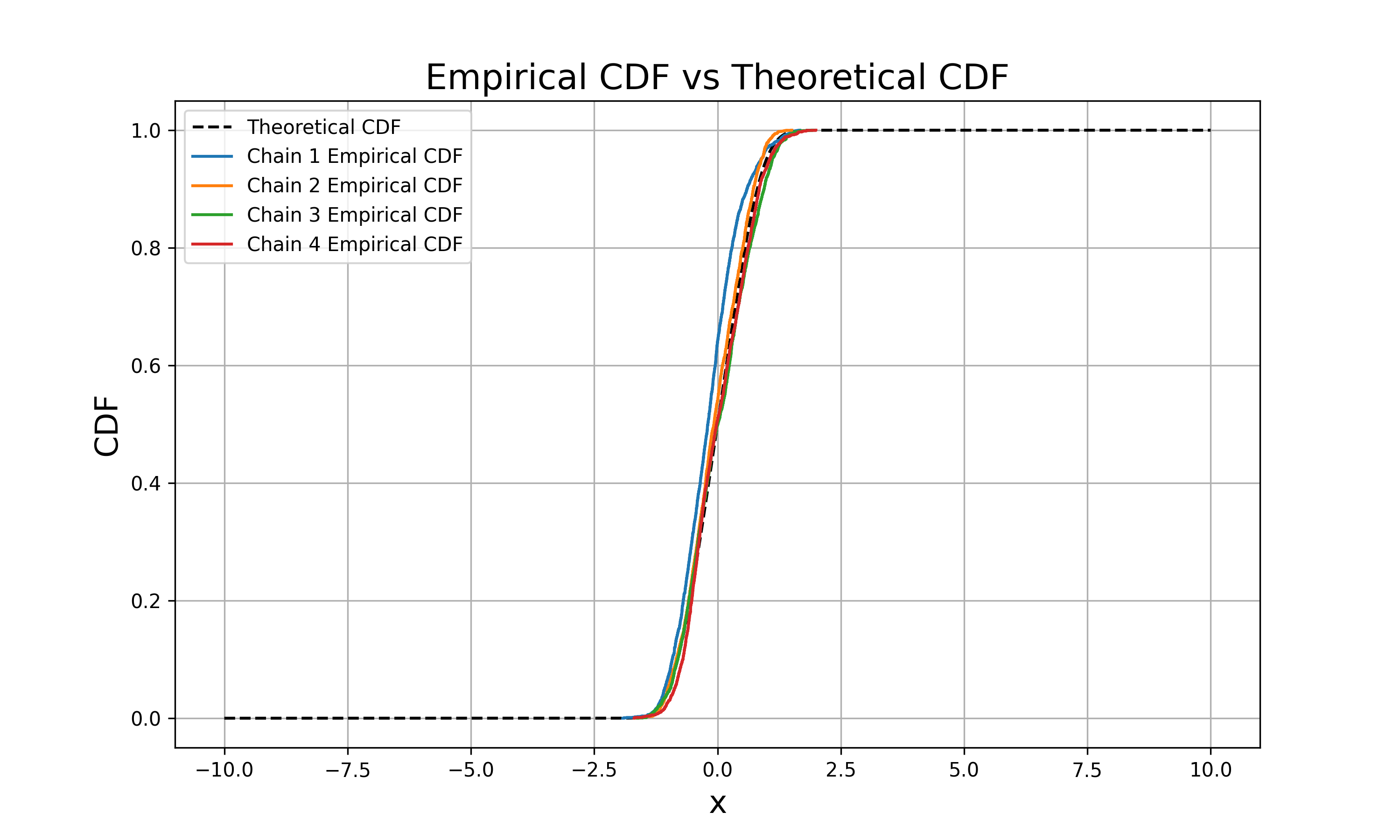}
	\caption{Comparison of the empirical PDFs of four Markov chains with the theoretical PDF of the generalized normal distribution.}
	\label{fig:cdf}
\end{figure*}

At last, in Figure \ref{fig:cdf}, we compare the empirical cumulative distribution functions (CDFs) for each chain using the optimal scale value $\tau^*=1.366$, against the theoretical CDF. 	The CDF of the generalized normal distribution is given by
$$
F(x; \mu, \alpha, \beta) = 
	\frac{1}{2} \left[1 + \frac{\text{sign}(x - \mu)}{ \Gamma\left(1/\beta\right)} \gamma\left(\frac{1}{\beta}, \left(\frac{|x - \mu|}{\alpha}\right)^\beta \right)\right],
$$
where $\gamma(\cdot)$ is the lower incomplete gamma function.
This comparison allows us to evaluate how effectively the chains sample from the cumulative distribution of the target. Notably, the empirical CDFs for the first dimension of each Markov chain closely align with the theoretical CDF.

\section*{Acknowledgments}
This research project was partially supported by NIH grant 1R21AI180492-01 and the Individual Research Grant at Texas A\&M University. 
The author would like to thank five anonymous reviewers and the Editors for their very constructive comments and efforts on this lengthy work, which greatly improved the quality of this paper.

\bibliography{bib-ms}

\end{document}